\documentclass[letterpaper,11pt]{amsart}

\usepackage{amsmath,amssymb,amsxtra,amsthm,hyperref}
\usepackage{tikz}



\newcommand{\bh}[1] {\mathcal{B}(\mathcal{#1})}
\newcommand{\pref}[1] {(\ref{#1})}

\newcommand{\lspan} {\operatorname{span}}
\newcommand*{\medcup}{\mathbin{\scalebox{1.2}{\ensuremath{\cup}}}}%
\newcommand{\sminus}{\scalebox{0.5}[1.0]{\( - \)}}

\newtheorem{theorem}{Theorem}[section]
\newtheorem{corollary}[theorem]{Corollary}
\newtheorem{proposition}[theorem]{Proposition}    
\newtheorem{lemma}[theorem]{Lemma}

\theoremstyle{definition}
\newtheorem{definition}[theorem]{Definition}

\newtheorem{example}[theorem]{Example}
\newtheorem{remark}[theorem]{Remark}

\numberwithin{equation}{section}

\begin{document}

\title[Regular Dilation and Nica-covaraint Representations]{Regular Dilation and Nica-covariant Representation on Right LCM Semigroups}

\author{Boyu Li}
\address{Pure Mathematics Department\\University of Waterloo\\Waterloo, ON\\Canada \ N2L--3G1}
\email{b32li@math.uwaterloo.ca}
\date{\today}


\begin{abstract} Regular dilation has recently been extended to graph product of $\mathbb{N}$, where having a $\ast$-regular dilation is equivalent to having a minimal isometric Nica-covariant dilation. In this paper, we first extend the result to right LCM semigroups, and establish a similar equivalence among $\ast$-regular dilation, minimal isometric Nica-covariant dilation, and a Brehmer-type condition. This result can be applied to various semigroups to establish conditions for $\ast$-regular dilation.
\end{abstract}

\maketitle

\section{Introduction}

Regular dilation is a special type of dilation result first studied by Brehmer in \cite{Brehmer1961}, as a generalization to the celebrated Sz.Nagy's dilation. It has since been studied by many authors \cite{Halperin1962, SFBook, Gaspar1997} and has been generalized to product systems \cite{Solel2008, Shalit2010} and lattice ordered semigroups \cite{BLi2014}. Recently, the author further generalized it to graph products of $\mathbb{N}$ \cite{Li2017}, which is a special class of quasi-lattice ordered groups \cite{CrispLaca2002,CrispLaca2007}. It has now become evident that there is a connection between regular dilation, and Nica-covariant representation. 

Isometric Nica-covariant representations were first studied by Nica \cite{Nica1992} on quasi-lattice ordered semigroups where he studied its $C^*$-algebra as a natural generalization of the Toeplitz-Cuntz algebra. It has been intensively studied since then\cite{LacaRaeburn1996, CrispLaca2002} and has been generalized to other classes of semigroups \cite{XLi2012}.

This paper fully characterizes all representations of right LCM semigroups that are $\ast$-regular by establishing a Brehmer type condition. Our result unifies many prior results on regular dilation. This includes Brehmer's theorem, Frazho-Bunce-Popescu's dilation of row contractions, regular dilation on lattice ordered semigroups, and regular dilation on graph products of $\mathbb{N}$. 

In Section \ref{sec.background}, we briefly go over the background of right LCM semigroups and Nica-covariant representations. In Section \ref{sec.regular}, we first extend the definition of $\ast$-regular dilation to the context of right LCM semigroups. We then establish the equivalence among having a $\ast$-regular dilation, having a minimal isometric Nica-covariant dilation, and a Brehmer type condition. The Brehmer type condition we obtain in Section \ref{sec.regular} requires for every finite subset $F$ of a right LCM semigroup $P$, certain operator $Z(F)\geq 0$. We would like to reduce it to a much smaller collection of finite subsets. 

In Section \ref{sec.desc}, we set up a few technical lemmas that help us reduce the positivity of the operator $Z(F)$ to other subsets. We then define the descending chain condition on right LCM semigroups, and showed that it suffices to verify the Brehmer type condition for all finite subset $F$ of the set of minimal elements when the right LCM semigroup satisfies the descending chain condition. 

Many well studied right LCM semigroups satisfy the descending chain condition. This includes the Artin monoids, the Thompson's monoid, $\mathbb{N}\rtimes\mathbb{N}^\times$, and the Baumslag-Solitar monoids. We derive their corresponding Brehmer type condition in Section \ref{sec.examples}.

Finally, in Section \ref{sec.gp}, we study representations of graph products of right LCM semigroups. We used the techniques that were developed in Section \ref{sec.desc} to reduce the Brehmer type condition in this context. We also consider an application on doubly commuting representations of direct sums of right LCM semigroups. 

\section{Background}\label{sec.background}

\subsection{Semigroups} 

A semigroup is a set with an associative multiplication. For our purpose, a semigroup $P$ is always unital, meaning that there exists $e\in P$ so that $x=xe=ex$ for all $x\in P$. Recently, there is a lot of research interest on the study of $C^*$-algebras associated with left cancellative semigroups \cite{XLi2016}. 

\begin{definition} A semigroup $P$ is called left cancellative if for any $p,a,b\in P$ with $pa=pb$, we have $a=b$. 
\end{definition} 

In this paper, we focus on a special class of left-cancellative semigroups that are called the right LCM semigroups. 

\begin{definition}
A unital semigroup $P$ is called right LCM if it is left cancellative and for any $p,q\in P$, either $pP\cap qP=rP$ for some $r\in P$ or $pP\cap qP=\emptyset$. 

We denote $P^\ast$ be the set of invertible elements in $P$. 
\end{definition} 

In the case when $pP\cap qP=rP$, we can treat $r$ as a least common right multiple of $p,q$ (and hence the name right LCM). There might be many such least common multiples, but it is clear that if $r,r'$ are both least common multiples of $p,q$, then there exists an invertible $u$ with $r\cdot u=r'$. For each $p,q\in P$, let us denote $p\vee q=\{r: pP\cap qP=rP\}$. Similarly, for a finite subset $F\subset P$, let $\vee F=\{r: \bigcap_{x\in F} xP = rP\}$. Notice that $\bigcap_{x\in F} xP=\emptyset$ if and only if $\vee F=\emptyset$.

Right LCM semigroups are considered a natural generalization of the well studied quasi-lattice ordered groups. Quasi-lattice ordered groups were first defined by Nica in \cite{Nica1992}, where he studied isometric covariant representations and their $C^*$-algebras. These representations are now known as isometric Nica-covariant representations, and they have been intensively studied since then \cite{LacaRaeburn1996, CrispLaca2002, CrispLaca2007, Starling2015}. 

Given a group $G$ and a unital semigroup $P\subseteq G$ with $P\cap P^{-1}=\{e\}$, the semigroup $P$ defines a partial order on $G$ via $x\leq y$ if $x^{-1}y\in P$. In other words, $x\leq y$ if there exists $p\in P$ with $y=xp$. One can check that $x\leq y$ if and only if $yP\subseteq xP$. This also defines a partial order on the semigroup $P$. Dually, the semigroup $P$ also defines a partial order $x\leq_r y$ if $yx^{-1} \in P$. Notice that for a right LCM semigroup $P$, we can similarly define an order $x\leq y$ if $x^{-1}y\in P$. However, this order is only a pre-order on the semigroup $P$ since $x\leq y\leq x$ implies $y=xu$ for some $u\in P^\ast$. 

\begin{definition} The partial order $\leq$ defined by $P$ on $G$ is called a quasi-lattice order if any finite set $F\subset G$ with an upper bound in $G$ has a least upper bound in $G$, denoted by $\vee F$. In this case, the pair $(G,P)$ is called a \emph{quasi-lattice ordered group}. We often refer to $P$ a quasi-lattice ordered semigroup. 
\end{definition}

In a quasi-lattice ordered group $(G,P)$, it is often convenient to add an element $\infty$ where $x\cdot \infty = \infty\cdot x = \infty$ for all $x\in G$. Then, $x\leq \infty$ for all $x\in G$. Therefore, whenever $F\subset G$ has no upper bound in $G$, we can denote $\vee F=\infty$. 

One may immediately notice that given a quasi-lattice ordered group $(G,P)$, the semigroup $P$ is right LCM since for any $x,y\in P$, either $x\vee y=z$ and $xP\cap yP=zP$, or $x\vee y=\infty$ and $xP\cap yP=\emptyset$. However, a right LCM semigroup is not necessarily a quasi-lattice ordered semigroup. Firstly, the set of invertible elements in a right LCM semigroup may not be $\{e\}$, in which case, there is no way to embed $P$ inside a group $G$. Secondly, even if $P^\ast=\{e\}$, it is often hard to check if one can embed $P$ inside a group $G$ (see for example, Artin monoids \cite{Paris2002}). Finally, even if $P$ embeds injectively inside a group $G$, it is hard to verify the partial order on $G$ defined by $P$ is a quasi-lattice order (though the partial order on $P$ is a `quasi-lattice order' due to the right LCM condition). 

\begin{example}\label{ex.quasi} The class of right LCM semigroups covers a wide range of examples.
\begin{enumerate}
\item $P$ is called an \emph{$\ell$-semigroup} (often called the lattice ordered semigroup \cite{LatticeOrderBookIntro}) if $P$ is normal inside $G$ so that every pair of elements in $G$ has a greatest lower bound and a least upper bound. In other words, the partial order on $G$ is a lattice order. A $\ell$-semigroup $P$ together with $G$ is always a quasi-lattice ordered group.
\item In \cite{CrispLaca2002}, there is another notion of \emph{lattice ordered group} where any pair of elements in $G$ has a least upper bound. It was shown in \cite[Lemma 27]{CrispLaca2002}, this is the same as $P$ being quasi-lattice ordered and $G=PP^{-1}$. This definition does not require the semigroup $P$ to be normal, and thus contains a wider class of examples. For example, we shall see that Artin monoids of finite type are important examples that fall under this class. Lattice ordered groups are quasi-lattice ordered.
\item\label{ex.quasi.1} $(\mathbb{Z}^k,\mathbb{N}^k)$ and the free group $(\mathbb{F}_k,\mathbb{F}_k^+)$ are both quasi-lattice ordered groups. Here, $(\mathbb{Z}^k,\mathbb{N}^k)$ is in fact an $\ell$-group, but the free group is not.
\item\label{ex.quasi.2} Given a simple graph $\Gamma$ on $k$ vertices, one can define $P_\Gamma$, the graph product associated with the graph to be the unital semigroup generated by $k$ generators where $e_i,e_j$ commute whenever there is an edge between the vertices $i,j$. This is also known as the graph product of $\mathbb{N}$. It is also called the right angled Artin monoid or the graph semigroup. It is a quasi-lattice ordered semigroup inside the group generated by the same set of generators. Notice that in the special case when the graph is the complete graph, $P_\Gamma$ is simply $\mathbb{N}^k$. When the graph contains no edge, $P_\Gamma$ is the free semigroup on $k$ generators. 
\item\label{ex.quasi.thompson} The Thompson's monoid is closely related to the well-known Thompson's group. There is a great interest in whether the Thompson's group is amenable or not. The Thompson's monoid can be written as $$F^+=\left<x_0,x_1,\cdots | x_nx_k=x_kx_{n+1}, k<n\right>.$$

The Thompson's monoid embeds injectively in the Thompson group, and it is a right LCM semigroup \cite{XLi2016} (it follows from the discussion after \cite[Lemma 6.32]{XLi2016} that every constructible right ideal of $F^+$ is principle and thus it has the right LCM property). 
\end{enumerate}
\end{example} 

An important class of right LCM semigroups that we shall focus on is the class of Artin monoids.

\begin{example}\label{ex.Artin} We first denote $\langle s,t\rangle_m=stst\cdots$, where we write $s,t$ alternatively for a total of $m$ times. For example, $\langle s,t\rangle_3=sts$.

Consider a symmetric $n\times n$ matrix $M$ where $m_{i,i}=1$ for all $i$, and $m_{i,j}\in\{2,\cdots,+\infty\}$ when $i\neq j$. One can define $A_M^+$, the Artin monoid associated with $M$ to be the unital semigroup generated by $e_1,\cdots,e_n$, where each $e_i,e_j$, $i\neq j$, satisfy the relation $\langle e_i, e_j\rangle_{m_{i,j}}=\langle e_j, e_i\rangle_{m_{i,j}}$. In particular, when $m_{i,j}=+\infty$, this means there is no relation between $e_i$ and $e_j$. One can similarly define the Artin group $A_M$ be the group generated by the same set of generators. 

The Artin monoid is said to be right-angled if each $m_{i,j}=2$ or $+\infty$ for all $i\neq j$. One may define a graph $\Gamma$ on $n$ vertices where $i,j$ are adjacent whenever $m_{i,j}=2$. The graph product associated with $\Gamma$ discussed in the Example \ref{ex.quasi} \pref{ex.quasi.2} is precisely the right-angled Artin monoid. 

The Artin monoid is said to be of finite type if each $m_{i,j}<\infty$. For example, if for all $i\neq j$, $m_{i,j}=3$ when $|i-j|=1$ and $m_{i,j}=2$ otherwise, then the Artin group is the familiar Braid group on $(n+1)$-strings. 

It is known that $(A_M, A_M^+)$ is a quasi-lattice ordered group when it is right angled or of finite type. In fact, these two cases are the only known examples to form a quasi-lattice ordered group \cite{CrispLaca2002}. However, it is known that the Artin monoid $A_M^+$ itself is always a right LCM semigroup (see \cite{XLi2016}). 
\end{example}

In \cite{BRRW}, it is shown that the Zappa-Sz\'{e}p product of semigroups provide a way to construct a rich class of right LCM semigroups. Let $A,U$ be two unital semigroup with identities $e_A,e_U$ respectively. Suppose there are two maps $A\times U\to U$ by $(a,u)\to a\cdot u$ and $A\times U\to A$ by $(a,u)\to a|_u$ that satisfy:
\begin{align*}
& (B1) e_A \cdot u = u;  & &(B5) a\cdot (uv)=(a\cdot u)(a|_u \cdot v);  \\
& (B2) (ab)\cdot u=a\cdot (b\cdot u);  & &(B6) a|_{uv}=(a|_u)|_v; \\
& (B3) a\cdot e_U = e_U; & &(B7) e_A|u=e_A; \\
& (B4) a|_{e_U}=a; & &(B8) (ab)|u=a|_{b\cdot u} b|_u. 
\end{align*}

Then the external Zappa-Sz\'{e}p product $U\bowtie A$ is the cartesian product $U\times A$ with multiplication defined by $$(u,a)(v,b)=(u(a\cdot v), (a|_v) b).$$

This allows us to build more right LCM semigroups from existing ones. 

\begin{lemma}[Lemma 3.3, \cite{BRRW}]\label{lm.zappa} Suppose $U,A$ are left cancellative semigroups with maps $(a,u)\to a\cdot u$ and $(a,u)\to a|_u$ that defines a Zappa-Sz\'{e}p product $U\bowtie A$. Suppose $U$ is a right LCM semigroup, and the set of constructible right ideals of $A$ is totally ordered by inclusion, and $u\to a\cdot u$ is a bijection from $U$ to $U$ for each $a\in A$. Then $U\bowtie A$ is a right LCM semigroup. 
\end{lemma}

Here, the set of constructible right ideals of left-cancellative semigroups is an important concept introduced by Xin Li in his work on semigroup $C^*$-algebras. One may refer to \cite{XLi2016} for more detail. 

\begin{example}\label{ex.Zappa} Zappa-Sz\'{e}p product provides more examples of right LCM semigroups.
\begin{enumerate}
\item\label{ex.Zappa.BS} Baumslag-Solitar monoids form another class of quasi-lattice ordered groups recently studied in \cite{Spielberg2012, CHR2016}. For $n,m\geq 1$, the Baumslag-Solitar monoid $B_{n,m}$ is the monoid generated by $a,b$ with the relation $ab^n=b^m a$. It is pointed out in \cite[Section 3.1]{BRRW} that they are Zappa-Sz\'{e}p product of $$U=\langle e, a, ba, \cdots, b^{m-1} a\rangle, A=\langle e, b\rangle.$$
\item\label{ex.Zappa.NN} The semigroup $\mathbb{N} \rtimes \mathbb{N}^\times$ where $$(x,a)(y,b)=(x+qy, ab).$$
One can similarly define $\mathbb{Q} \rtimes \mathbb{Q}^\times_+$. It is known that the pair $(\mathbb{Q} \rtimes \mathbb{Q}^\times_+, \mathbb{N} \rtimes \mathbb{N}^\times)$ is quasi-lattice ordered \cite[Proposition 2.1]{LacaRaeburn2010}. It is also shown in \cite[Section 3.2]{BRRW} that this semigroup is a Zappa-Sz\'{e}p product. 
\item One can construct a right LCM semigroup that is not quasi-lattice ordered using Zappa-Sz\'{e}p product. Take $U=\mathbb{N}^\times$ and $A=\mathbb{T}$, and let $a\cdot u=u$, $a|_u=a^u$ for all $a\in A, u\in U$. Their Zappa-Sz\'{e}p product can be described as $$(n,e^{i\alpha})(m,e^{i\beta})=(nm, e^{i(m\alpha+\beta)}).$$

One can easily verify that $U\bowtie A$ is a right LCM semigroup using the Lemma \ref{lm.zappa}. In $U\bowtie A$, $(1,1)$ is the identity. Moreover, the set of invertible elements consists of $(1,e^{i\alpha})$, where the inverse of $(1,e^{i\alpha})$ is $(1,e^{-i\alpha})$. Since it has non-trivial invertible elements, $U\bowtie A$ cannot be a quasi-lattice ordered semigroup (since we cannot embed it injectively inside a group). 
\end{enumerate}
\end{example}

We now briefly discuss a few important properties of right LCM semigroups which will be useful later. For the rest of this section, we fix a right LCM semigroup $P$. 

Let $a\in P$ and let $F\subset P$ be a finite subset. Denote $a\cdot F=\{a\cdot p: p\in F\}$. If $bP\supseteq\bigcap_{x\in F} xP$, we often write $b^{-1} \vee F=\{b^{-1} r: r\in \vee F\}$. Notice that since $bP\supseteq\bigcap_{x\in F} xP$, for each $r\in \vee F$, $bP\supseteq rP$ and $r=bp$ for some $p\in P$. This implies that $b^{-1}r\in P$ and $b^{-1} \vee F\subset P$, even though $b^{-1}$ is not part of the semigroup. 

\begin{lemma}\label{lm.left.inv} Let $a\in P$ and $F\subset P$ be a finite subset, $\vee\left( a\cdot F\right) = a\cdot \vee F$.
\end{lemma} 

\begin{proof} It suffices to show $\bigcap_{x\in F} axP = a\cdot \bigcap_{x\in F} xP$. The containment $\supseteq$ is obvious. For the $\subseteq$ direction, take $r\in \bigcap_{x\in F} axP$ and let $F=\{x_1,\cdots,x_n\}$. We can find $p_1,\cdots,p_n\in P$ so that $r=ax_i p_i$. By the left cancellative property, $x_i p_i=x_jp_j$ for all $i,j$, and thus $r\in a \cdot \bigcap_{x\in F} xP$. 
\end{proof} 

Let $F$ be a finite subset of $P$ and $x\in \vee F$. Consider the set $x\vee y$ for some $y\in P$. Notice for any $s\in\vee F$, $x=su$ for some invertible element $u\in P^\ast$. Therefore, $xP=suP=sP$ and thus $$x\vee y=\{r: rP=xP\cap yP\}=\{r: rP=sP\cap yP\}.$$

Therefore, $x\vee y$ is independent on the choice of $x\in\vee F$. For simplicity, we shall write it as $(\vee F)\vee y$. 

\begin{lemma}\label{lm.union} Let $F_1,F_2\subset P$ be two finite sets. Then $$\vee (F_1\cup F_2) = \left(\vee F_1\right) \vee \left(\vee F_2\right).$$
\end{lemma} 

\begin{proof} Fix $s_i\in\vee F_i$, we have 
\begin{align*}
\left(\vee F_1\right) \vee \left(\vee F_2\right) &= s_1 \vee s_2 \\
&= \{r: rP=s_1P\cap s_2 P\} \\
&= \{r: rP=\Big(\bigcap_{x\in F_1} xP\Big)\cap \Big(\bigcap_{x\in F_2} xP\Big) \} \\
&= \{r: rP=\bigcap_{x\in F_1\cup F_2} xP\} \\
&= \vee (F_1\cup F_2).
\end{align*}

The argument still works when one of $\vee F_i=\emptyset$. \end{proof}

\begin{lemma}\label{lm.pull} Let $p_1,\cdots,p_n\in P$ and $a\in P$. Let $F_1=\{p_1\cdot a, p_2, \cdots, p_n\}$ and $F_2=\{a, p_1^{-1}(p_1\vee p_2), \cdots, p_1^{-1}(p_1\vee p_n)\}$. Then $$\vee F_1 = p_1\cdot \vee F_2.$$
\end{lemma} 

\begin{proof} Take $s_i\in p_1\vee p_i$ for all $2\leq i\leq n$. Since $s_i\in p_1P$, $p_1^{-1} s_i\in P$ for all $i$. 

If $s_i'\in p_1\vee p_i$, then $s_i=s_i' u$ for some invertible $u$, and thus $s_i P= s_i' P$. Therefore, $\vee F_2=\vee \{a, p_1^{-1} s_i\}$.  Hence, by Lemma \ref{lm.left.inv},
\begin{align*}
p_1 \cdot \vee F_2 &= p_1\cdot \vee \{a, p_1^{-1} s_i\} \\
&= \vee \big(p_1\cdot \{a, p_1^{-1} s_i\}  \big)\\ 
&= \vee \{p_1 a, s_i\}
\end{align*}

But $s_iP=p_1P\cap p_i P$ since $s_i\in p_1\vee p_i$. Therefore,
\begin{align*}
p_1 \cdot \vee F_2 &= \vee \{p_1 a, s_i\} \\
&= \big\{r: rP=p_1aP \cap \big(\bigcap_{i=2}^n s_iP \big)\big\} \\
&= \big\{r: rP=p_1aP \cap \big(\bigcap_{i=2}^n p_1P\cap p_i P \big)\big\} \\
&= \big\{r: rP=p_1aP \cap \big(\bigcap_{i=1}^n p_i P \big)\big\}
\end{align*}

Notice that $p_1 aP\subseteq p_1 P$, and thus $$\big\{r: rP=p_1aP \cap \big(\bigcap_{i=1}^n p_i P \big)\big\} = \big\{r: rP=p_1aP \cap \big(\bigcap_{i=2}^n p_i P \big)\big\} = \vee F_1.\qedhere$$ \end{proof} 

\subsection{Nica-covariance}

The Nica-covariance representations play a central role in our study of regular dilation. The study of isometric Nica-covariant representations was originated from Nica's work on certain covariant representations of the quasi-lattice ordered groups and their $C^*$-algebras \cite{Nica1992}. It has since been generalized to right LCM semigroups, and more recently, to left cancellative semigroups via constructible ideals based on Xin Li's work \cite{XLi2016}.

\begin{definition}\label{df.NC} Given a right LCM semigroup $P$, a representation $V:P\to\bh{K}$ is called an isometric Nica-covariant representation if for each $p\in P$, $V(p)\in\bh{K}$ is an isometry, and for any $p,q\in P$, we have 
\begin{align*}V(p)V(p)^*V(q)V(q)^*=
\begin{cases} 
V(r)V(r)^*, & \mbox{ if }rP=pP\cap qP\\
0, & \mbox{ otherwise.}
\end{cases}
\end{align*}
\end{definition} 

Notice that for any invertible element $u\in P^\ast$, $V(u)$ is a unitary. When $rP=sP=pP\cap qP$, we can find an invertible $u\in P^\ast$ so that $r=su$. Hence, $$V(r)V(r)^*=V(su)V(su)^*=V(s)V(s)^*.$$
Therefore, the Nica-covariance condition is well defined on right LCM semigroups. 

In the case of a quasi-lattice ordered group $(G,P)$, it is often convenient to add $\infty$ to the quasi-lattice ordered semigroup $P$ and define $p\cdot\infty=\infty=\infty\cdot p$ for all $p\in P$. Moreover, for any $p,q\in P$ without a common upper bound, we may define $p\vee q=\infty$. Under this notation, an isometric Nica-covariant representation is a representation $V:P\to\bh{K}$ with $V(\infty)=0$ and $V(p)V(p)^*V(q)V(q)^*=V(p\vee q)V(p\vee q)^*$ for all $p,q\in P$. 

\begin{example} The Nica-covariance condition can be checked easily in many cases.
\begin{enumerate} 
\item Consider the semigroup $\mathbb{N}^k$ inside $\mathbb{Z}^k$. A representation $V:\mathbb{N}^k\to\bh{H}$ is uniquely determined by its image on $k$ generators $V_i=V(e_i)$. The Nica-covariance condition implies that for all $i\neq j$, $$V_i V_i^* V_j V_j^* = V_i V_j V_i^* V_j^*.$$
By multiplying $V_i^*$ on the left and $V_j$ on the right, we obtain $V_i^* V_j=V_j V_i^*$. Since $V$ is a representation, $V_i, V_j$ commute with one another. Therefore, we see that $\{V_1,\cdots,V_k\}$ is a family of doubly commuting isometries. In fact, every such family of doubly commuting isometries defines an isometric Nica-covariant representation on $\mathbb{N}^k$.
\item Consider the free semigroup $\mathbb{F}_k^+$. A representation $V:\mathbb{N}^k\to\bh{H}$ is uniquely determined by its image on $k$ generators $V_i=V(e_i)$. Since for all $i\neq j$, $e_i \vee e_j=\infty$, we have $$V_i V_i^* V_j V_j^* = 0.$$
Therefore, $V_i^* V_j=0$, which means $\{V_1,\cdots,V_k\}$ is a family of isometries with orthogonal ranges. Every such family defines an isometric Nica-covariant representation on $\mathbb{F}_k^+$. 
\end{enumerate} 
\end{example} 

\section{Regular Dilation}\label{sec.regular}

The study of dilation theory started when Sz.Nagy proved the celebrated Sz.Nagy dilation theorem, which states that for any contraction $T\in\bh{H}$, there exists an isometry $V\in\bh{K}$ with $\mathcal{H}\subset\mathcal{K}$ co-invariant, so that $T=P_\mathcal{H} V|_\mathcal{H}$. Soon, Ando \cite{Ando1963} extended Sz.Nagy's dilation to a pair of commuting contractions. However, extension to more commuting contractions fails due to an example of Parrott \cite{Parrott1970} where he gave a triple of commuting contractions that fails to have a commuting dilation. 

There are many ways to generalize Sz.Nagy's dilation beyond a pair of commuting contractions. Brehmer first studied regular dilations of commuting contractions in \cite{Brehmer1961}. Frazho-Bunce-Popescu also studied dilations of non-commutative row contractions. Recently, Brehmer's dilation and Frazho-Bunce-Popescu's dilation were unified as regular dilation on graph products of $\mathbb{N}$ \cite{Li2017}. It also turns out that having regular dilation in these cases corresponds to having a minimal isometric Nica-covariant dilation \cite{Li2017}. This motivates us to extend regular dilation further to representations of right LCM semigroups.  

Suppose $T:P\to\bh{H}$ is a representation of a semigroup $P$. Suppose there exists a larger Hilbert space $\mathcal{K}\supset\mathcal{H}$ and an isometric representation $V:P\to\bh{K}$ so that for all $p\in P$, $V(p)$ is an isometry and, $$T(p)=P_\mathcal{H} V(p)\big|_\mathcal{H}.$$

This representation $V$ of $P$ is called an isometric dilation of $T$. A result of Sarason \cite{Sarason1966} states that we can decompose $\mathcal{K}=\mathcal{H}_-\oplus\mathcal{H}\oplus\mathcal{H}^+$, so that under such decomposition, the isometric dilation $V(p)$ has the form: $$V(p)=\begin{bmatrix} * & 0 & 0 \\ * & T(p) & 0 \\ * & * & *\end{bmatrix}$$

$V$ is called minimal if $$\mathcal{K}=\overline{\lspan}\{V(p)h: p\in P,h\in\mathcal{H}\}$$
When $V$ is minimal, $\mathcal{H}_-$ must be $\{0\}$ and thus $\mathcal{H}$ is co-invariant for $V$. For each $p\in P$, we can write $V(p)$ as a $2\times 2$ block matrix with respect to the decomposition $\mathcal{K}=\mathcal{H}\oplus\mathcal{H}^\perp$: $$V(p)=\begin{bmatrix} T(p) & 0 \\ * & * \end{bmatrix}.$$
Notice that when $V$ is a minimal dilation of $T$, $\|T(p)\|\leq \|V(p)\|=1$, and thus $T$ is always a contractive representation. Throughout this paper, we assume every representation of semigroups is contractive. 

We are interested in the case when $T$ has a minimal isometric dilation $V$ that is also Nica-covariant. In such case, we say $V$ is a minimal isometric Nica-covariant dilation of $T$.  

A common tool in studying regular dilation is the completely positive definite kernel \cite{Popescu1996,Popescu1999b}. Given a unital semigroup $P$, a unital Toeplitz kernel is a map $K:P\times P\to\bh{H}$ so that $K(e,e)=I$, $K(p,q)=K(q,p)^*$, and $K(ap,aq)=K(p,q)$. It is called completely positive definite if for any $p_1,\cdots,p_n\in P$, the operator matrix $[K(p_i,p_j)]\geq 0$. A classical result known as the Naimark dilation theorem \cite{Naimark1943} can be restated as the following theorem (\cite[Theorem 3.2]{Popescu1999b}):

\begin{theorem}\label{thm.Naimark} If $K:P\times P\to\bh{H}$ is a completely positive definite kernel, then there exists a Hilbert space $\mathcal{K}\supset\mathcal{H}$ and an isometric representation $V:P\to\bh{K}$ so that $$K(p,q)=P_\mathcal{H} V(p)^* V(q)\big|_\mathcal{H} \mbox{ for all }p,q\in P.$$
Moreover, there is a unique minimal dilation $V$, up to unitary equivalence, that satisfies 
$$\overline{\lspan}\{V(p)h: p\in P,h\in\mathcal{H}\}=\mathcal{K},$$
and $\mathcal{H}$ is co-invariant for $V$. The minimal dilation $V$ is called the Naimark dilation of $K$. 

Conversely, if $V:P\to\bh{K}$ is a minimal isometric dilation of $T:P\to\bh{H}$, then let $$K(p,q)=P_\mathcal{H} V(p)^* V(q)\big|_\mathcal{H}$$
We have $K(p,q)$ is a completely positive definite Toeplitz kernel with $K(e,p)=T(p)$ for all $p\in P$.
\end{theorem}

\begin{proof} The proof of the theorem can be found in \cite[Theorem 3.2]{Popescu1999b}. However, it is worthwhile to briefly go over the proof since it explicitly constructs the minimal Naimark dilation that are useful later. 

First let $\mathcal{K}_0=P\otimes \mathcal{H}$ and define a degenerate inner product by $$\left\langle \sum \delta_p\otimes h_p, \sum \delta_q\otimes k_q\right\rangle = \sum_{p,q} \langle K(q,p)h_p,k_q\rangle.$$

Let $\mathcal{N}=\{k\in\mathcal{K}_0:\langle k,k\rangle=0\}$ and $\mathcal{K}$ be the completion of $\mathcal{K}_0/\mathcal{N}$ with respect to the inner product. $\mathcal{H}$ is naturally embedded in $\mathcal{K}$ as $\delta_e\otimes\mathcal{H}$. For each $p\in P$, define $V(p)\delta(q)\otimes h=\delta(pq)\otimes h$. One can check $V:P\to\bh{K}$ is the minimal Naimark dilation of $T$. 

For the converse, it is simple to check that $K$ is indeed a Toeplitz kernel. To show it is completely positive definite, take any $p_1,\cdots,p_n\in P$, the operator matrix
\begin{align*}
 &[K(p_i,p_j)] \\
=&P_{\mathcal{H}^n} [V(p_i)^* V(p_j)]\big|_{\mathcal{H}^n} \\
=&P_{\mathcal{H}^n} \left(\begin{bmatrix} V(p_1)^* \\ \vdots \\ V(p_n)^*\end{bmatrix} \begin{bmatrix} V(p_1) & \cdots & V(p_n)\end{bmatrix}\right) \Big|_{\mathcal{H}^n}\geq 0.
\end{align*}

Therefore, $K$ is a completely positive definite Toeplitz kernel. Moreover, $K(e,p)=P_\mathcal{H} V(p)\big|_\mathcal{H}=T(p)$ for all $p\in P$. 
\end{proof}

Given a contractive representation $T:P\to\bh{H}$: if there is a completely positive definite kernel $K$ so that for all $p\in P$, $K(e,p)=T(p)$, then the minimal Naimark dilation of $K$ is a minimal isometric dilation of $T$. It follows from Theorem \ref{thm.Naimark} that every minimal isometric dilation for $T$ corresponds to a minimal Naimark dilation of a completely positive definite Toeplitz kernel $K$. 

However, it is often difficult to find such a kernel $K$ explicitly from $T$. For example, take $P=\mathbb{N}^2$ whose representation is determined by a pair of commuting contractions $T_i=T(e_i)$. It follows from the Ando's dilation that $T$ has an isometric dilation $V$ and thus there exists a completely positive definite kernel $K$ with $K(e,p)=T(p)$. However, one can hardly ever write out $K$ explicitly.

This motivated the study of regular dilation where we can write a kernel $K$ based on $T$ and ask whether such kernel is completely positive definite. Suppose $T$ has a minimal isometric Nica-covariant dilation $V$, then the Toeplitz kernel $K$ defined by $V$ can be written out in terms of $T$ in an explicit way.

\begin{proposition}\label{pn.regular.NC} Let $T:P\to\bh{H}$ be a contractive representation of a right LCM semigroup $P$ so that it has a minimal isometric Nica-covariant dilation $V$. Then, the completely positive definite Toeplitz kernel $K$ associated with $V$ is given by $$K(p,q)=P_\mathcal{H} V(p)^* V(q) \big|_\mathcal{H}=T(p^{-1}s) T(q^{-1}s)^*$$ for all $p,q\in P$, $s\in p\vee q$. 
\end{proposition} 

\begin{proof} By the Nica-covariance, $V(p)V(p)^*V(q)V(q)^*=V(s) V(s)^*$ for $s\in p\vee q$, where by convention, $V(s)=0$ if $p\vee q=\emptyset$. Multiplying $V(p)^*$ on the left and $V(q)$ on the right gives us $$V(p)^* V(q)=V(p^{-1} s)V(q^{-1} s)^*.$$

Since $V$ is minimal, $\mathcal{H}$ is co-invariant. With respect to the decomposition $\mathcal{K}=\mathcal{H}\oplus\mathcal{H}^\perp$, each $V(a)$ can be written as $$V(a)=\begin{bmatrix} T(a) & 0 \\ * & * \end{bmatrix}.$$

Therefore, for any $a,b\in P$, $V(a) V(b)^*$ can be written as $$V(a)V(b)^* = \begin{bmatrix} T(a)T(b)^* & * \\ * & * \end{bmatrix}.$$

Therefore,
\begin{align*}
K(p,q) &= P_\mathcal{H} V(p)^* V(q)\big|_\mathcal{H} \\
&= P_\mathcal{H} V(p^{-1}s)V(q^{-1} s)^*\big|_\mathcal{H} \\
&= T(p^{-1}s)T(q^{-1} s)^*
\end{align*}

This proves the desired result. \end{proof} 

This motivates our definition of $\ast$-regular dilation. 

\begin{definition}\label{df.regular} Let $P$ be a right LCM semigroup and $T:P\to\bh{H}$ a unital contractive representation. Define a Toeplitz kernel $K:P\times P\to\bh{H}$ by $$K(p,q)=T(p^{-1}s) T(q^{-1}s)^*$$ for all $p,q\in P$, $s\in p\vee q$. Here, we assume by convention that when $p\vee q=\emptyset$, $K(p,q)=0$. 

We say $T$ has a $\ast$-regular dilation if this kernel $K$ is completely positive definite. In such case, the minimal Naimark dilation $V$ of the kernel $K$ is called the $\ast$-regular dilation of $T$. 
\end{definition}

\begin{remark} This kernel $K$ is well defined since for any $s,t\in p\vee q$, there exists an invertible $u$ with $s=tu$. Therefore, $$T(p^{-1} s)T(q^{-1} s)^*=T(p^{-1} t)T(q^{-1} t)^*.$$
\end{remark} 

\begin{remark}\label{rm.regular.Toeplitz} This kernel $K$ is indeed a Toeplitz kernel. It is clear that $K(e,e)=I$, $K(p,q)=K(q,p)^*$. If $a\in P$, by Lemma \ref{lm.left.inv}, we have $ap\vee aq=a(p\vee q)$ and therefore $(ap)^{-1}(ap\vee aq)=p^{-1}(p\vee q)$ and similarly $(aq)^{-1}(ap\vee aq)=q^{-1}(p\vee q)$. 
\end{remark}

It is now evident from the Proposition \ref{pn.regular.NC} that the kernel in the Definition \ref{df.regular} is our only choice for $T$ to have a minimal isometric dilation that is also Nica-covariant. We shall soon see that the converse is also true (Theorem \ref{thm.reg.NC}): if this kernel $K$ is completely positive definite, then the minimal Naimark dilation is Nica-covariant. We first note that our definition of $\ast$-regular dilation coincides with the definition in the context of $\ell$-semigroup and graph product of $\mathbb{N}$. 

\begin{example} In the case that $P$ is an $\ell$-semigroup, regular dilation was first defined and studied in \cite{DFK2014} and a necessary and sufficient condition was given in \cite{BLi2014}. In such case, for every $p,q\in P$, there exists a unique pair $g_+,g_-\in P$ with $p^{-1} q=g_-^{-1} g_+$ and $g_-\wedge g_+=e$. The definition of $\ast$-regularity on an $\ell$-semigroup is equivalent to the kernel $K(p,q)=T(g_+)T(g_-)^*$ being completely positive definite. 

In fact, $g_+=(p\wedge q)^{-1} q=p^{-1}(p\vee q)$ and $g_-=(p\wedge q)^{-1} p = q^{-1}(p\vee q)$, and it is clear that these two definitions coincide. 

Historically, Brehmer's original definition of regular dilation on $\mathbb{N}^k$ requires the kernel $K(p,q)=T(g_-)^* T(g_+)$ to be completely positive, which is equivalent to $T^*$ being $\ast$-regular. This is why we adopt the notion of $\ast$-regular dilation instead of regular dilation. 
\end{example}

\begin{example} In the case that $P$ is a graph product of $\mathbb{N}$, $\ast$-regular dilation was recently defined in \cite{Li2017} as a generalization of the Brehmer dilation and Frazho-Bunce-Popescu dilation. The definition of $\ast$-regular dilation in this case can be summarized as follow: given $p,q\in P$, one first identifies the largest $a\in P$ so that $p=a\cdot p',q=a\cdot q'$ via repeatedly removing a common initial syllable. This procedures ends when there is no $e\neq b\in P$ with $p'=b\cdot p''$ and $q'=b\cdot q''$. Then the kernel is defined as $$K(p,q)=K(p',q')=\begin{cases} T(q')T(p')^*, \mbox{ if }p',q'\mbox{ commute}; \\
0, \mbox{ otherwise}.\end{cases}$$

Now if $p', q'$ do not commute, then $p'\vee q'=\infty$ and similarly $p\vee q=\infty$. Otherwise, since they have no common initial syllable, $p'\vee q'=p'q'$. Therefore,
\begin{align*}
p^{-1}(p\vee q) &= p'^{-1}(p'\vee q') \\
&= p'^{-1} p'\cdot q' = q' 
\end{align*}
Similarly, $q^{-1}(p\vee q)=p'$. Again, the Definition \ref{df.regular} coincides with that in \cite{Li2017}.
\end{example}  

\begin{theorem}\label{thm.reg.NC} $T$ has a $\ast$-regular dilation if and only if it has a minimal isometric Nica-covariant dilation. 
\end{theorem}

\begin{proof} It follows from the Proposition \ref{pn.regular.NC} that if $V$ is the minimal isometric Nica-covariant dilation, then for any $p,q\in P$ and $s\in p\vee q$,
$$K(p,q) = T(p^{-1}s)T(q^{-1} s)^* = P_\mathcal{H} V(p)^* V(q)\big|_\mathcal{H}.$$

Since $V$ is a minimal isometric dilation of $T$, it follows from the second half of Theorem \ref{thm.Naimark} that $K$ is completely positive definite, which is exactly how we defined $T$ to have a $\ast$-regular dilation. 

Conversely, suppose that $T$ has a $\ast$-regular dilation so that the kernel $K$ in the Definition \ref{df.regular} is completely positive definite. Let $V:P\to\bh{K}$ be the minimal Naimark dilation as constructed in the proof of Theorem \ref{thm.Naimark}. We first show that for any $p,q\in P$ and $s\in p\vee q$, $V(p)^* V(q)|_\mathcal{H}=V(p^{-1} s) V(q^{-1} s)^*|_\mathcal{H}$. Since $\lspan\{V(r)h:r\in P, h\in\mathcal{H}\}$ is dense in $\mathcal{K}$, it suffices to prove for any $r\in P$ and $h,k\in\mathcal{H}$, we have $$\langle V(p)^* V(q) h, V(r) k\rangle = \langle V(p^{-1} s) V(q^{-1} s)^* h, V(r) k\rangle.$$

Starting from the left hand side:
\begin{align*}
\langle V(p)^* V(q) h, V(r) k\rangle &= \langle  V(pr)^* V(q) h,  k\rangle \\
&= \langle  K(pr,q) h, k\rangle_\mathcal{H} \\
&= \langle  T((pr)^{-1} t) T(q^{-1} t)^* h, k\rangle_\mathcal{H}
\end{align*}

Here, $t\in pr \vee q$. In the special case when $pr\vee q=\emptyset$, $K(pr,q)=0$ and $V(pr)^* V(q)=0$ by the Nica-covariance condition, and thus the result follows. Otherwise, $pr\vee q\neq\emptyset$ and thus $p\vee q\neq\emptyset$. 

Take $s\in p\vee q$ and $w=p^{-1} s$. Notice now, by the Lemma \ref{lm.pull},
$$p^{-1}(pr \vee q) = r \vee (p^{-1}s) = r\vee w .$$
Hence,
\begin{align*}
q^{-1}(pr \vee q) &= q^{-1}s \cdot s^{-1}(pr\vee q) \\
&= q^{-1} s \cdot w^{-1} p^{-1}(pr \vee q) \\
&= q^{-1} s\cdot w^{-1} (r\vee w).
\end{align*}
Therefore, take $v=p^{-1} t\in r\vee w$, 
\begin{align*}
 & \langle  T((pr)^{-1} t) T(q^{-1} t)^* h, k\rangle_\mathcal{H} \\
=& \langle  T(r^{-1} v) T(w^{-1} v)^* T(q^{-1}s)^* h, k\rangle_\mathcal{H} \\
=& \langle K(r, w) V^*(q^{-1}s)^* h, k\rangle_\mathcal{H} \\
=& \langle V(p^{-1}s)V^*(q^{-1}s)^* h, V(r) k\rangle.
\end{align*}
Here, we used the fact that for all $s\in P$, $\mathcal{H}$ is co-invariant for $V$ and thus $h'=V^*(q^{-1}s)^* h\in\mathcal{H}$. Since $h',k\in\mathcal{H}$, $$\langle K(r,w) h', k\rangle = \langle V(r)^* V(w) h', k\rangle = \langle V(w) h', V(r) k\rangle.$$

Now it suffices to show for all $r\in P$ and $s\in p\vee q$, $$V(p)^* V(q) V(r)|_\mathcal{H}=  V(p^{-1} s) V(q^{-1}s)^* V(r)|_\mathcal{H}.$$ 

Denote $w=q^{-1}s$ and similar to the computation earlier, observe that $w\vee r=q^{-1}(p\vee qr)$. Take $t\in p\vee qr$ and $v=q^{-1}t\in w\vee r$, and start from the left, 
\begin{align*}
V(p)^* V(q) V(r)|_\mathcal{H} =& V(p^{-1} t) V(r^{-1}q^{-1} t)^*|_\mathcal{H} \\
=& V(p^{-1}s) V(w^{-1} v) V^*(r^{-1} v)|_\mathcal{H} \\
=& V(p^{-1}s) V(w)^* V(r)|_\mathcal{H} \\
=& V(p^{-1}s) V(q^{-1}s)^* V(r)|_\mathcal{H}
\end{align*}

This proves for any $p,q\in P$ and $s\in p\vee q$, $V(p)^* V(q)=V(p^{-1}s) V(q^{-1}s)^*$. Multiplying $V(p)$ on the left and $V(q)^*$ on the right proves that $V$ is Nica-covariant. \end{proof}

It has been observed that the kernel $K$ being completely positive is often equivalent to a Brehmer-type condition where a collection of operators (instead of a collection of operator matrices) are positive. This is the case in Brehmer's dilation, Frazho-Bunce-Popescu's dilation, and more recently, dilation on graph products of $\mathbb{N}$. We first establish a Brehmer-type condition in the case of an arbitrary right LCM semigroup.

For simplicity, we shall denote $TT^*(p)=T(p)T(p)^*$. It is clear that $TT^*(pq)=T(p)TT^*(q)T(p)^*$. Since $T$ is contractive, for each invertible $u\in P^*$, $T(u)$ must be an unitary. For a finite subset $F\subset P$, we define $TT^*(\vee F)=TT^*(p)$ for some $p\in\vee F$. This is well-defined since for any two $p,q\in\vee P$, $p=qu$ for some invertible $u\in P^*$. Therefore, $TT^*(p)=T(q)TT^*(u)T(q)^*=TT^*(q)$.

\begin{theorem}\label{thm.main} Let $T:P\to\bh{H}$ be a unital representation of a right LCM semigroup. The following are equivalent:
\begin{enumerate}
\item\label{thm.main.1} $T$ has a $\ast$-regular dilation;
\item\label{thm.main.2} $T$ has a minimal isometric Nica-covariant dilation;
\item\label{thm.main.3} For any finite set $F\subset P$, $$Z(F)=\sum_{U\subseteq F} (-1)^{|U|} TT^*(\vee U) \geq 0.$$
\end{enumerate}
\end{theorem}
 
\begin{proof} First of all, the equivalence between \pref{thm.main.1} and \pref{thm.main.2} is shown in the Theorem \ref{thm.reg.NC}. 

To show \pref{thm.main.2} implies \pref{thm.main.3}, let $V:P\to\bh{K}$ be the minimal isometric Nica-covariant dilation for $T:P\to\bh{H}$. Consider the product $\prod_{p\in F}(I-V(p)V(p)^*)$: notice that for any subset $U\subseteq F$, by the Nica-covariance, $$\prod_{p\in U} V(p)V(p)^* = V(\vee U)V(\vee U)^*.$$

Hence, $$\prod_{p\in F}(I-V(p)V(p)^*)=\sum_{U\subseteq F} (-1)^{|U|} V(\vee U)V(\vee U)^*.$$

Now since $\mathcal{H}$ is co-invariant for $V$, we have $$P_\mathcal{H} V(\vee U)V(\vee U)^* \big|_\mathcal{H} = T(\vee U) T(\vee U)^*.$$

By restricting to $\mathcal{H}$, we have
\begin{align*}
Z(F) &=\sum_{U\subseteq F} (-1)^{|U|} T(\vee U) T(\vee U)^* \\
&= P_\mathcal{H}\Big(\sum_{U\subseteq F} (-1)^{|U|} V(\vee U) V(\vee U)^*\Big)\Big|_\mathcal{H} \\
&= P_\mathcal{H}\Big(\prod_{p\in F}(I-V(p)V(p)^*)\Big)\Big|_\mathcal{H} \geq 0
\end{align*}

Now to show \pref{thm.main.3} implies \pref{thm.main.1}, it suffices to show for any $F_0=\{p_1,\cdots,p_n\}\subset P$, the operator matrix $K[F_0]$ is positive. Now for each $U\subseteq F_0$, pick $s_U\in \vee U$. Since $$\vee \{p_i\}=\{r:rP=p_i P\}=p_i P^\ast,$$ we can pick $s_{p_i}=p_i$ for all $i$. Let $F_1=\{s_U: U\subseteq F_0\}$. $F_1$ is still a finite subset of $P$, and  $F_0\subset F_1$ since each $s_{p_i}=p_i\in F_1$. Therefore, it suffices to show $K[F_1]\geq 0$. Let us now show that $K[F_1]\geq 0$ given the condition \pref{thm.main.3}.

First, rows and columns of $K[F_1]$ are indexed by subsets of $F_0$. For any subsets $A_i,A_j\subseteq F_0$, the $(A_i,A_j)$-entry of $K[F_1]$ can be expressed as
$$K(s_{A_i}, s_{A_j}) = T\left(s_{A_i}^{-1} s \right) T\left(s_{A_j}^{-1} s\right)^* $$
for some $s\in s_{A_i}\vee s_{A_j}$ (the choice does not affect the value). By Lemma \ref{lm.union}, $$s_{A_i\cup A_j}\in \vee(A_i\cup A_j)=(\vee A_i)\vee (\vee A_j)=s_{A_i}\vee s_{A_j}.$$ 

Hence,
$$K(s_{A_i}, s_{A_j}) = T\left(s_{A_i}^{-1} s_{A_i\cup A_j} \right) T\left(s_{A_j}^{-1} s_{A_i\cup A_j}\right)^* $$

Now define an operator matrix $R$ with the same dimension as $K[F_1]$. For any subsets $A_i, A_j\subseteq F_0$, define the $(A_i,A_j)$-entry of $R$ to be $0$ if $A_i$ is not a subset of $A_j$. Otherwise, define $R(A_i,A_j)$ to be:
$$T\left(s_{A_i}^{-1} s_{A_j}\right)\Big(\sum_{A_j\subseteq U\subseteq F_0} (-1)^{|U\backslash A_j|} TT^*\left(s_{A_j}^{-1} s_U\right)\Big)^{1/2} $$

We first show that this is well defined given the Condition \pref{thm.main.3}. For a fixed $A\subseteq F_0$, let $F_0\backslash A=\{q_1,\cdots,q_k\}$ and define 
$$F_A=\{s_A^{-1} s_{A\cup\{q_j\}}: 1\leq j\leq k\}.$$

Then,
\begin{align*}
Z(F_A) =& \sum_{W\subseteq F_A} (-1)^{|W|} TT^*(\vee W) \\
=& \sum_{W_0\subseteq F_0\backslash A} (-1)^{|W_0|} TT^*\Big(\bigvee_{q\in W_0} s_A^{-1} s_{A\cup\{q\}}\Big) \\
=& \sum_{W_0\subseteq F_0\backslash A} (-1)^{|W_0|} TT^*\Big(s_A^{-1} \left(\vee (A\cup W_0)\right)\Big) \\
=& \sum_{A\subseteq U\subseteq F_0} (-1)^{|U\backslash A|} TT^*\left(s_A^{-1} (\vee U)\right) \\
=& \sum_{A\subseteq U\subseteq F_0} (-1)^{|U\backslash A|} TT^*\left(s_A^{-1} s_U \right)
\end{align*}

Therefore, $R(A_i,A_j)$ is in fact equal to $T\left(s_{A_i}^{-1} s_{A_j}\right) Z(F_{A_j})^{1/2}$, where $Z(F_{A_j})\geq 0$ by the Condition \pref{thm.main.3}. We now claim that $$K[F_1]=R\cdot R^*\geq 0.$$ 

Fix $A_i, A_j\subset F_0$ for which we compute the $(A_i,A_j)$-entry of $R\cdot R^*$, which is equal to $\sum_{U\subseteq F_0} R(A_i, U) R(A_j, U)^*$. By the construction of $R$, $R(A_i,U)R(A_j,U)^*\neq 0$ only when $A_i, A_j$ are subsets of $U$. Therefore, 
\begin{align*}
RR^*[A_i,A_j] = &\sum_{U\subseteq F_0} R(A_i, U) R(A_j, U)^* \\
=& \sum_{A_i\cup A_j\subseteq U\subseteq F_0} R(A_i, U) R(A_j, U)^* \\
=& \sum_{A_i\cup A_j\subseteq U\subseteq F_0} T\left(s_{A_i}^{-1} s_U\right) Z(F_U) T\left(s_{A_j}^{-1} s_U\right)^*
\end{align*}

Replacing $Z(F_U)$ using the earlier computation, we obtain 
\begin{align*}
&RR^*[A_i,A_j] \\
=&\sum_U\sum_{U\subseteq W} (-1)^{|W\backslash U|} T\left(s_{A_i}^{-1} s_U\right) 
    TT^*\left(s_U^{-1} s_W\right)
    T\left(s_{A_j}^{-1} s_U\right)^* \\
=& \sum_U\sum_{U\subseteq W} (-1)^{|W\backslash U|} T\left(s_{A_i}^{-1} s_W\right) 
    T\left(s_{A_j}^{-1} s_W\right)^*
\end{align*}

Consider the term $T\left(s_{A_i}^{-1} s_W\right) T\left(s_{A_j}^{-1} s_W\right)^*$ in the double summation. It occurs whenever $A_i\cup A_j \subseteq U \subseteq W$. Let $m=\left|W\backslash(A_i\cup A_j)\right|$ and $k=|W\backslash U|$, $U$ has to contain all the elements in $A_i\cup A_j$ and $m-k$ elements in $W\backslash(A_i\cup A_j)$. There are precisely ${m \choose k}$ choices of $U$. Therefore,
\begin{align*}
 & RR^*[A_i,A_j] \\
=& \sum_{W: U\subseteq W} \left(\sum_{k=0}^m (-1)^{k} {m\choose k}\right) T\left(s_{A_i}^{-1} s_W\right) 
    T\left(s_{A_j}^{-1} s_W\right)^*
\end{align*}

Notice that $$\sum_{k=0}^m (-1)^{k} {m\choose k}=\begin{cases} 1,\mbox{ if }m=0; \\0,\mbox{ otherwise}.\end{cases}$$

Hence, the only non-zero term in the summation occurs when $m=0$ and thus $W_0=A_i\cup A_j$. Therefore,
\begin{align*}
 & \sum_{U\subseteq F_0} R(A_i, U) R(A_j, U)^* \\
=& T\left(s_{A_i}^{-1} s_{W_0} \right) 
    T\left(s_{A_j}^{-1} s_{W_0}\right)^*\\
=& T\left(s_{A_i}^{-1} s_{A_i\cup A_j} \right) 
    T\left(s_{A_j}^{-1} s_{A_i\cup A_j}\right)^*\\
=& K(s_{A_i}, s_{A_j})
\end{align*}

This finishes the proof. \end{proof}

\begin{remark} As observed in \cite{BLi2014, Li2017}, the matrix $R$ is a Cholesky decomposition of the operator matrix $K[F_1]$. Given two subsets $A_i, A_j\subseteq F_0$, $R(A_i, A_j)=0$ whenever $|A_j|>|A_i|$. When $|A_j|=|A_i|$, the only case when $R(A_i, A_j)\neq 0$ is when $A_i\subseteq A_j$ and thus $A_i=A_j$. Hence by arranging $F_1=\{\vee A: A\subseteq F_0\}$ according to $|A|$ in the decreasing order, the matrix $R$ becomes a lower triangular matrix. 
\end{remark} 

As a quick corollary, every co-isometric representation of a lattice ordered semigroup is $\ast$-regular. This generalizes \cite[Corollary 3.8]{BLi2014} in the case of $\ell$-semigroups. 

\begin{corollary} Suppose that $P$ is a lattice ordered semigroup in the sense that any finite subset of $P$ has a least upper bound. If $T:P\to\bh{H}$ is a co-isometric representation (i.e. $T(p)T(p)^*=I$ for all $p\in P$), then $T$ is $\ast$-regular. 
\end{corollary} 

\begin{proof} It suffices to check that $T$ satisfies the Condition \pref{thm.main.3} in the Theorem \ref{thm.main}. For any finite set $F\subset P$ and any $U\subseteq F$, since $P$ is lattice ordered, $\vee U\in P$ and thus $T(\vee U)T(\vee U)^*=I$. Therefore,
$$Z(F)=\sum_{U\subseteq F} (-1)^{|U|} T(\vee U) T(\vee U)^* = \sum_{U\subseteq F} (-1)^{|U|} I= 0.\qedhere$$
\end{proof} 

\section{Descending Chain Condition}\label{sec.desc}

In general, Condition \pref{thm.main.3} in the Theorem \ref{thm.main} can be very difficult to verify since it requires $Z(F)\geq 0$ for all finite subset $P$. Our goal is to reduce it to a smaller collection of finite subsets.

\subsection{Reduction Lemmas} We first prove a few technical lemmas that helps us with the reduction.

\begin{lemma}\label{lm.base} Let $F\subseteq P$ be a finite subset.
\begin{enumerate}
\item\label{lm.base.1} If $F=\{p_1,p_2,\cdots,p_n\}$ where $p_1P=p_2P$, then let $F_0=\{p_2,\cdots,p_n\}$. Then $Z(F)=Z(F_0)$ and thus $Z(F)\geq 0$ if and only if $Z(F_0)\geq 0$.
\item\label{lm.base.2} If  $F=\{p_1,p_2,\cdots,p_n\}$ and $p_1\in P^\ast$, then $Z(F)=0$. 
\end{enumerate}
\end{lemma}

\begin{proof}  For \pref{lm.base.1}: Consider $Z(F)=\sum\limits_{U\subseteq F} (-1)^{|U|} TT^*(\vee U)^*$. For any $U_0\subseteq\{p_3,\cdots,p_n\}$ and consider the terms $U_1=\{p_1\}\cup U_0$ and $U_2=\{p_1,p_2\}\cup U_0$. Since $p_1P=p_2P$, it is clear that $\vee U_1=\vee U_2$ and $|U_2|=|U_1|+1$. Therefore, $$(-1)^{|U_1|} TT^*(\vee U_1)+(-1)^{|U_1|} TT^*(\vee U_2)=0.$$
Hence,
$$Z(F)-Z(F_0) = \sum_{p_1\in U\subset F} (-1)^{|U|} TT^*(\vee U)^* = 0.$$

For \pref{lm.base.2}: Since $p_1\in P^\ast$ is invertible, $p_1 P=P$. Hence, for any $U_0\subseteq\{p_2,\cdots,p_n\}$, $\vee U_0 = \vee \{p_1\}\cup U_0$. It follows from a similar argument that $Z(F)=0$. 
\end{proof} 

\begin{lemma}\label{lm.reduction} Let $T:P\to\bh{H}$ be a unital representation of a right LCM semigroup. Let $p_1,\cdots,p_n,q\in P$. Define:
\begin{align*}
F &= \{p_1\cdot q, p_2, \cdots, p_n\}, \\
F_1 &= \{p_1,\cdots, p_n\}, \\
F_2 &= \{q, p_1^{-1} s_2, \cdots, p_1^{-1} s_n\}. \\
\end{align*}
where $s_i\in p_1\vee p_i$ for all $2\leq i\leq n$. Here, when $p_1\vee p_i=\emptyset$, we can exclude the term $p_1^{-1} s_i$ in $F_2$. 

Then $Z(F)=Z(F_1)+T(p_1) Z(F_2) T(p_1)^*$. 
In particular, $Z(F)\geq 0$ if $Z(F_1),Z(F_2)\geq 0$. 
\end{lemma} 

\begin{proof}  Let $F_0=\{p_2,\cdots,p_n\}$ and consider $Z(F)-Z(F_1)$:
\begin{align*}
 &Z(F)-Z(F_1)\\
=& \sum_{U\subseteq F} (-1)^{|U|} TT^*(\vee U)
- \sum_{U\subseteq F_1} (-1)^{|U|} TT^*(\vee U)
\end{align*}

The only difference between $F$ and $F_1$ is their first element, and therefore the only difference between $Z(F)$ and $Z(F_1)$ occurs when $U$ contains the first element. Hence,
\begin{align*}
 & Z(F)-Z(F_1) \\
=& \sum_{U\subseteq F_0} (\!\sminus 1\!)^{|U|+1} \left(TT^*\left(\vee (\{p_1 q\}\medcup U)\right)
   - TT^*\left(\vee (\{p_1 \}\medcup U)\right) \right) \\
=& T(p_1) \big(\!\sum_{U\subseteq F_0} (\!\sminus 1\!)^{|U|+1} TT^*\!\big(p_1^{\sminus 1}\vee\!(\{p_1 q\}\!\medcup\! U)\big)
   \!-\! TT^*\!\big(p_1^{\sminus 1} \vee\!(\{p_1\}\!\medcup\! U)\big)\!\big) T(p_1)^* \\
=& T(p_1) \big(\!\sum_{U\subseteq F_0} (\!\sminus 1\!)^{|U|+1} TT^*\!\big(q\vee\!\bigvee_{p\in U} p_1^{\sminus 1}(p_1\!\vee\! p)\big)
   \!-\! TT^*\!\big(\bigvee_{p\in U} p_1^{\sminus 1}(p_1\!\vee\! p)\big)\!\big) T(p_1)^* \\
=& T(p_1) \big(\sum_{q\in U\subseteq F_2} (\!\sminus 1\!)^{|U|} TT^*\left(\vee U\right)
   + \sum_{q\notin U\subseteq F_2} (\!\sminus 1\!)^{|U|} TT^*\left(\vee U \right) \big) T(p_1)^* \\
=& T(p_1) Z(F_2) T(p_1)^*
\end{align*}

Now it is clear that $Z(F)\geq 0$ if $Z(F_1)\geq 0$ and $Z(F_2)\geq 0$. In the case when $p_1\vee p_i=\emptyset$, $\vee U=\emptyset$ whenever $p_1,p_i\in U\subset F_1$ or $p_1q, p_i\in U\subset F$. Therefore, we can simply pretend that the term $p_1^{-1} s_i$ does not exist in $F_2$. The calculation will not be affected. \end{proof}

\begin{remark}\label{rm.reduction}
Lemma \ref{lm.reduction} allows us to reduce the positivity of $Z(F)$ to the positivity of $Z(F_1),Z(F_2)$. $F_1$ replaces the element $p_1 q\in F$ by $p_1\in F_1$ while keeping the rest of it unchanged. Moreover, since $p_1 q P\subseteq p_1 P$, take $r_1\in\vee F_1$ and $r\in\vee F$, we have $rP\subseteq r_1P$ and thus $r=r_1 v$ for some $v\in P$. For $F_2$, observe that
\begin{align*}
\vee F &= \left(p_1 q \vee p_2 \vee \cdots \vee p_n \vee e\right) \\
&= p_1 \cdot \left(q \vee \left(p_1^{-1} (p_1\vee p_2)\right) \vee \cdots \left(p_1^{-1} (p_1\vee p_n)\right)\right) \\
&= p_1 \cdot \vee F_2
\end{align*}

Intuitively, elements are `smaller' in $F_1,F_2$ compared to $F$. 
\end{remark}

\begin{remark} In the case when $T$ is an isometric Nica-covaraint representation, 
$$Z(F)=(I-TT^*(p_1q)) \cdot \prod_{i=2}^n (I-TT^*(p_i))$$

Observe that $$I-TT^*(p_1q)=(I-TT^*(p_1)) + T(p_1)(I-TT^*(q))T(p_1)^*.$$

Therefore,
\begin{align*}
Z(F)=& (I-TT^*(p_1)) \cdot \prod_{i=2}^n (I-TT^*(p_i)) \\
 & + T(p_1)\left((I-TT^*(q))\cdot \prod_{i=2}^n (I-TT^*(p_i))\right) T(p_1)^*\\
=& Z(F_1) + T(p_1) Z(F_2) T(p_1)^*.
\end{align*}
\end{remark} 

\subsection{Ore semigroup} We say the right LCM semigroup $P$ is an Ore semigroup if for any $p,q\in P$, $pP\cap qP\neq\emptyset$. In the case of quasi-lattice ordered group, this corresponds to the lattice order condition discussed in \cite{CrispLaca2002} where every finite subset $F$ of $P$ always has a least upper bound. 

\begin{definition} We say that $P$ satisfies \emph{the descending chain condition} if there is no infinite sequence $x_n\in P$ and $y_n\notin P^\ast$ so that $x_n=x_{n+1} y_n$ or $x_n=y_n x_{n+1}$.

An element $x\in P$ is called \emph{minimal} if $x\notin P^\ast$ and whenever $x=yz$ for $y,z\in P$, either $y\in P^\ast$ or $z\in P^\ast$. We let $P_{min}$ be the set of all minimal elements in $P$.
\end{definition} 

Intuitively, $P$ has the descending chain property if we cannot take way non-invertible factors from each $x\in P$ from the left or the right infinitely many times. 

\begin{remark} In the case when $(G,P)$ is a quasi-lattice ordered group, the descending chain condition is saying there is no infinite sequence $x_n$ so that $x_{n+1} < x_n$ (i.e. when there is $y_n\neq e$, $x_n=x_{n+1} y_n$) or $x_{n+1} <_r x_n$ (i.e. when there is $y_n\neq e$, $x_n= y_n x_{n+1}$). We are not sure if the descending chain property of the partial order $<$ (or $<_r$) alone would be sufficient. 
\end{remark} 

Suppose $P$ satisfies the descending chain condition, it is clear that $P_{min}\neq\emptyset$ since otherwise we can build an infinite descending chain starting from any element $x\neq e$. It turns out that testing subsets of $P_{min}$ is sufficient for Condition \pref{thm.main.3} in the Theorem \ref{thm.main}. 

\begin{proposition}\label{pn.desc.reduction} Let $P$ be a right LCM Ore semigroup that satisfies the descending chain condition. Suppose $Z(F)\geq 0$ for all finite $F\subset P_{min}$. Then $Z(F)\geq 0$ for all finite $F\subset P$. 
\end{proposition}  

\begin{proof} Pick any finite $F\subset P$. If $F\cap P^\ast\neq\emptyset$, we have $Z(F)=0\geq 0$ by Lemma \ref{lm.base}. If $F\nsubseteq P_{min}$, we can pick some element $x\in F$ that is not minimal. Therefore, we can write $x=p_1\cdot q$ for $p_1,q\notin P^\ast$ and write $F=\{p_1q,p_2,\cdots,p_n\}$. We have $Z(F)\geq 0$ if $Z(F_1)\geq 0$ and $Z(F_2)\geq 0$ where $F_1,F_2$ are defined in the Lemma \ref{lm.reduction}. 

This process allows us to build a binary tree rooted at $F$. Let $\mathbb{F}_2^+$ be the free semigroup generated by $\{1,2\}$, and let $\epsilon\in\mathbb{F}_2^+$ be the empty word. We start with $F_\epsilon=F$. Suppose for a word $\omega\in\mathbb{F}_2^+$ where $F_\omega\nsubseteq P_{min}\cup P^\ast$, we can pick an element $x=p_1\cdot q\in F_\omega$ where $p_1,q\notin P^\ast$. This allows us to define $F_{\omega 1}$ and $F_{\omega 2}$ as in the Lemma \ref{lm.reduction}. We have $Z(F_\omega)\geq 0$ whenever $Z(F_{\omega 1})\geq 0$ and $Z(F_{\omega 2})\geq 0$. 

Suppose the binary tree is finite, its leaves contain finite subsets $\overline{F}\subset P_{min}\cup P^\ast$. We know such $\overline{F}$ satisfies $Z(\overline{F})\geq 0$ by the hypothesis (in the case when $\overline{F}\subset P_{min}$) or the Lemma \ref{lm.base} (in the case when $\overline{F}\cap P^\ast\neq\emptyset$). Therefore, it suffices to show the binary tree is finite. 

Assume otherwise that the binary tree is infinite. By the K\"{o}nig Lemma, this tree has an infinite path $s_1s_2\cdots s_n\cdots$, $s_i\in\{1,2\}$. Let $\omega_n=s_1s_2\cdots s_n$ so that $F_{\omega_n}$ are nodes in the binary tree. Pick $t_{\omega}\in \vee F_{\omega}$ for each node of the binary tree. As we observed in the Remark \ref{rm.reduction}, there exists $p_{\omega 2}\notin P^\ast$ so that $p_{\omega 2}\cdot t_{\omega 2} = t_{\omega}$ and some element $u_{\omega 1}\in P$ so that $t_{\omega 1} u_{\omega 1} = t_{\omega}$. By the descending chain condition, this implies there is only finitely many $s_i=2$ and hence there is $N$ so that $s_i=1$ for all $i>N$. 

For $n>N$, the only difference between $F_{\omega_n}$ and $F_{\omega_{n+1}}=F_{\omega_n 1}$ is an element $p_1 q\in F_{\omega_n}$ and $p_1\in F_{\omega_{n+1}}$ where $q\notin P^\ast$. By the descending chain condition again, this process cannot continue infinitely many times. This proves the binary tree has to be finite which finishes the proof. \end{proof} 

As an immediate consequence, we can replace the Condition \pref{thm.main.3} in the Theorem \ref{thm.main} by a much smaller collection of subsets when the semigroup has the descending chain property. 

\begin{theorem}\label{thm.main.desc}  Let $T:P\to\bh{H}$ be a unital representation of a right LCM Ore semigroup with the descending chain property. Let $P_{min}$ be the set of all minimal elements in $P$. The following are equivalent:
\begin{enumerate}
\item\label{thm.main.desc.1} $T$ has a $\ast$-regular dilation;
\item\label{thm.main.desc.2} $T$ has a minimal isometric Nica-covariant dilation;
\item\label{thm.main.desc.3} For any finite set $F\subset P_{min}$, $$Z(F)=\sum_{U\subseteq F} (-1)^{|U|} TT^*(\vee U)^* \geq 0.$$
\end{enumerate}
\end{theorem}

\subsection{Non-Ore Semigroup}

In the case of a right LCM semigroup that fails to satisfy the Ore condition, the proof of the Proposition \ref{pn.desc.reduction} fails due to the fact that $\vee F$ can be $\emptyset$. Nevertheless, a similar argument can be applied.

\begin{definition} We say a subset $P_0$ of a right LCM semigroup a minimal set if 
\begin{enumerate}
\item $P_{min}\subseteq P_0$
\item For any $x\in P_{min}$ and $y\in P_0$, we have $$x^{-1}(x\vee y)\subseteq P_0\cup P^\ast.$$ 
\end{enumerate}
\end{definition} 

It is clear that $P_0=P$ is always a minimal set. However, in many cases, we can choose $P_0$ to be a much smaller set. 

\begin{proposition}\label{pn.desc.reduction.2} Let $P$ be a right LCM semigroup that satisfies the descending chain condition. Let $P_0$ be any minimal set of $P$. Suppose $Z(F)\geq 0$ for all finite $F\subset P_0$. Then $Z(F)\geq 0$ for all finite $F\subset P$. 
\end{proposition}  

\begin{proof} For every finite $F\subset P$, denote $m(F)=|F\cap P_0|$ which counts the number of elements in $F$ that are from $P_0$. In the case when $m(F)=|F|$, we have $F\subset P_0$ and thus $Z(F)\geq 0$. Otherwise, we will show that we can find a collection $F_1,\cdots, F_k$ with $m(F_i)>m(F)$, and $Z(F)\geq 0$ whenever $Z(F_i)\geq 0$ for all $i$. This allows us to proceed with induction with $m(F)$.

Suppose $m(F)<|F|$, pick $x\in F$ so that $x$ is not in $P_0$. Since $P$ has the descending chain condition, we can repeatedly remove a minimal element from $x$ for a finite number of times. Hence, we can write $x=x_1x_2\cdots x_n$ where $x_i\in P_{min}$. Write $F=\{x,p_2,p_3,\cdots,p_n\}$. Apply Lemma \ref{lm.reduction}, $Z(F)\geq 0$ if $Z(F_1),Z(F_2)\geq 0$, where 
\begin{align*}
F_1 &= \{x_1,p_2,p_3,\cdots,p_n\} \\
F_2 &= \{x_2x_3\cdots x_n, x_1^{-1}(x_1\vee p_2),\cdots, x_1^{-1}(x_1\vee p_n)\} 
\end{align*} 

Notice that $x_1\in P_{min}\subset P_0$, and thus $m(F_1)=m(F)+1$. For each $p_i\in F\cap P_0$, $x_1^{-1}(x_1\vee p_i)\in P_0\cup P^\ast$. If $x_1^{-1}(x_1\vee p_i)\in P^\ast$, then it follows from Lemma \ref{lm.base} that $Z(F_2)=0$. Otherwise, we must have $m(F_2)\geq m(F)$. In the case when $n=2$, $x_n\in P_{min}\subset P_0$ and we get $m(F_2)>m(F)$, which we can proceed with induction. Otherwise, notice that though $m(F_2)=m(F)$, the element $x=x_1\cdots x_n\in F$ is replaced by $x'=x_2x_3\cdots x_n$ in $F_2$, where $x'$ is a product of $(n-1)$ minimal elements. Repeat the same procedure again for $F_2$, we get $Z(F_2)\geq 0$ if $Z(F_{21})\geq 0$ and $Z(F_{22})\geq 0$, where $m(F_{21})>Z(F_2)\geq Z(F)$ and $m(F_{22})\geq m(F_2)\geq Z(F)$. The inequality is strict when $n=3$ since $x_3\in F_{22}\cap P_0$. Otherwise, repeat the same procedure again. Eventually, we can reduce the positivity of $Z(F)$ to the positivity of $Z(F_i)$ with $m(F_i)>m(F)$. This finishes the proof. \end{proof} 

We now reach a nice condition for $\ast$-regularity in the case of an arbitrary right LCM semigroup with descending chain condition. 

\begin{theorem}\label{thm.main.desc.NO}  Let $T:P\to\bh{H}$ be a unital representation of a right LCM with the descending chain condition. Let $P_0$ be a minimal set. The following are equivalent:
\begin{enumerate}
\item\label{thm.main.desc.NO.1} $T$ has a $\ast$-regular dilation;
\item\label{thm.main.desc.NO.2} $T$ has a minimal isometric Nica-covariant dilation;
\item\label{thm.main.desc.NO.3} For any finite set $F\subset P_0$, $$Z(F)=\sum_{U\subseteq F} (-1)^{|U|} TT^*(\vee U)^* \geq 0.$$
\end{enumerate}
\end{theorem}

\section{Examples}\label{sec.examples}

We now examine several classes of right LCM semigroups that satisfy the descending chain condition. For each class of semigroups, we derive the corresponding conditions for $\ast$-regularity. 

\subsection{Artin Monoids} 

Artin monoids (see Example \ref{ex.Artin}) form an important class of right LCM semigroups. Their Nica-covariant representations and related $C^\ast$-algebras are studied in \cite{CrispLaca2002}. In the case of finite type or right-angled Artin monoids $P_M$, it is known that they are embedded injectively in the corresponding Artin group $G_M$, and $(G_M,P_M)$ form a quasi-lattice ordered group \cite{CrispLaca2002}. In general, Artin monoids are shown to embed injectively inside the corresponding artin group \cite{Paris2002}. The semigroup $P_M$ is known to be a right LCM semigroup, but it is unknown whether $(G_M,P_M)$ is quasi-lattice ordered.

Let $\{e_1,\cdots,e_n\}$ be the set of generators for $P_M$. Each element $p\in P_M$ can be written as $p=e_{i_1}e_{i_2}\cdots e_{i_n}$, and we define the length of $p$ to be $\ell(p)=n$ when $p$ can be expressed as a product of $n$ generators. Though there may be multiple ways to express $p$ as a product of generators, the relations on an Artin monoid are always homogeneous and thus it always takes the same number of generators to express $p$. Therefore, $\ell(p)$ is well-defined. 

\begin{lemma}\label{lm.Artin} Every Artin monoid $P_M$ has the descending chain property. The set of minimal elements is precisely the set of generators $\Gamma$. 
\end{lemma}

\begin{proof} Once we defined the length of an element $\ell(p)$ to be the number of generators requires to express $p$. We have for any $p,q\in P_M$, $\ell(pq)=\ell(p)+\ell(q)$. It is clear that we can not find infinite sequences $x_n$ and $y_n\neq e$ with $x_n=y_n x_{n+1}$ or $x_n=x_{n+1} y_n$ since otherwise, $\ell(x_n)\in\mathbb{Z}_{\geq 0}$ is strictly decreasing. 

Its set of minimal elements are precisely the set of elements with length $1$, which is exactly the set of generators. 
\end{proof}

The Artin monoids of finite types are all lattice ordered. Therefore, Theorem \ref{thm.main.desc} applies. 

\begin{theorem}\label{thm.main.Artin} A contractive representation $T$ of finite-type Artin monoids are $\ast$-regular if and only if $Z(F)\geq 0$ for all finite subset $F$ of the set of generators. 
\end{theorem} 

\begin{example} Let us consider the Braid monoid on 3 strands: $$B_3^+=\langle e_1,e_2: e_1e_2e_1=e_2e_1e_2\rangle.$$

A representation $T:B_3^+\to\bh{H}$ is uniquely determined by $T_i=T(e_i)$, $i=1,2$, which satisfies $T_1T_2T_1=T_2T_1T_2$. Theorem \ref{thm.main.Artin} states that $T$ is $\ast$-regular if and only if $T_1,T_2$ are contractions, and
\begin{align*}
&I-T_1T_1^*-T_2T_2^*+T(e_1\vee e_2)T(e_1\vee e_2)^* \\
=& I-T_1T_1^*-T_2T_2^*+T_1T_2T_1T_1^*T_2^*T_1^* \geq 0.
\end{align*}
\end{example} 

When the Artin monoid is infinite, it is hard to find a minimal set in general. Recall that an Artin monoid is called right-angled if entries in $M$ are either $2$ or $\infty$. This is also known as the graph product of $\mathbb{N}$. Regular dilations onn right-angled Artin monoids were studied in \cite{Li2017}. 

\begin{proposition}\label{prop.raArtin.min} Given a right-angled Artin monoid $A_M^+$, the set of generators $P_{min}=\{e_1,\cdots,e_n\}$ is also a minimal set. 
\end{proposition} 

\begin{proof} Pick any $e_i\in P_{min}$ and $e_j$ with $e_j\neq e_i$. Either $m_{ij}=2$, in which case $e_i^{-1}(e_i\vee e_j)=e_i^{-1}e_ie_j=e_j$. Or $m_{ij}=\infty$, in which case $e_i\vee e_j=\infty$. In either case, we can see $P_{min}$ is a minimal set. 
\end{proof} 

\begin{remark} $\ast$-regular dilation on graph products of $\mathbb{N}$ was characterized in \cite{Li2017}, which can be recovered from the Proposition \ref{prop.raArtin.min} combined with the Theorem \ref{thm.main.desc.NO}.
\end{remark} 

\subsection{Thompson's Monoid}

Recall the Thompson's monoid from Example \ref{ex.quasi} \pref{ex.quasi.thompson}: $$F^+=\left<x_0,x_1,\cdots | x_nx_k=x_kx_{n+1}, k<n\right>.$$

Our result of $\ast$-regular dilation can help us generate isometric Nica-covariant representations for the Thompson's monoid. We first show that $F^+$ has the descending chain property. 

\begin{lemma}\label{lm.Thompson} Thompson's monoid $F^+$ has the descending chain property. The set of minimal elements is the set of generators $\{x_0,x_1,\cdots\}$. The set of generators is also a minimal set for $F^+$. 
\end{lemma}

\begin{proof} Similiar to the case of Artin monoids, since the relations that define the Thompson's monoid $F^+$ are homogeneous, we can define $\ell(p)=n$ if we can write $p$ as a product of $n$ generators $p=x_{i_1}x_{i_2}\cdots x_{i_n}$. It is clear that for all $p,q\in F^+$, $\ell(p)+\ell(q)=\ell(pq)$. Therefore, $F^+$ has the descending chain property (otherwise, we can obtain a strictly decreasing sequence of $\ell(p_n)$). It is clear that the set of minimal elements are precisely the set of generators. 

Now for any $x_i,x_j$, $i<j$. It follows from the relation $x_j x_i= x_i x_{j+1}$ that $x_i\vee x_j=x_jx_i$ and thus both $x_i^{-1} (x_i\vee x_j)=x_j$ and $x_j^{-1}(x_i\vee x_j)=x_{j+1}$ are again minimal elements. Therefore, $P_{min}$ is also a minimal set.  
\end{proof} 

Again, the Theorem \ref{thm.main.desc} applies to the Thompson's monoid.

\begin{theorem}\label{thm.main.Thompson} Let $T:F^+\to\bh{H}$ be a unital representation uniquely determined by the generators $T_i=T(e_i)$. Then $T$ has a $\ast$-regular dilation if and only if for any finite subset $F$ of the generators, $Z(F)\geq 0$. 
\end{theorem} 

\subsection{$\mathbb{N} \rtimes \mathbb{N}^\times$} 

Recall the semigroup $\mathbb{N} \rtimes \mathbb{N}^\times$ (Example \ref{ex.Zappa} \pref{ex.Zappa.NN}) is the monoid $\{(a,p): a\in\mathbb{N}, p\in\mathbb{N}^\times\}$ with the multiplication $$(a,p)(b,q)=(a+bp,pq).$$

It embeds in $\mathbb{Q} \rtimes \mathbb{Q}^\times$, and they form a quasi-lattice ordered group \cite{LacaRaeburn2010}. The semigroup $\mathbb{N} \rtimes \mathbb{N}^\times$ has $(0,1)$ as the identity, and it is generated by $P_0=\{(1,1),(0,p): p\mbox{ is a prime}\}$ with the relations:
\begin{align*}
(0,p)(1,1) &= (p,p) = (1,1)^p (0,p), \\
(0,p)(0,q) &= (0,pq).
\end{align*}

It is obvious that $\mathbb{N} \rtimes \mathbb{N}^\times$ has the descending chain property and the set of minimal elements are precisely the set of its generators $P_{min}$. However, it is not a Ore-semigroup. For example, consider the principal right ideal generated by $(0,2)$ and $(1,2)$. For all $(b,q)\in P$, $(i,2)(b,q)=(i+2b,q)$, and thus the first coordinate always has the same parity as $i$. Therefore, $(0,2)P\cap (1,2)P=\emptyset$. In general, given $(a,m),(b,n)\in\mathbb{N} \rtimes \mathbb{N}^\times$, one can compute \cite[Remark 2.3]{LacaRaeburn2010}:
$$(a,m)\vee (b,n)=\begin{cases} (\ell, \operatorname{lcm}(m,n)): (a+m\mathbb{N})\cap(b+n\mathbb{N})\neq\emptyset; \\ \infty, (a+m\mathbb{N})\cap(b+n\mathbb{N})=\emptyset \end{cases}$$
Here, $\ell=\min\{(a+m\mathbb{N})\cap(b+n\mathbb{N})\}$. 

\begin{proposition} Let $P_0=\{(1,1),(i,p): 0\leq i<p, p\mbox{ is a prime}\}$. Then $P_0$ is a minimal set. 
\end{proposition} 

\begin{proof} We need to show for all $x\in P_{min}$ and $y\in P_0$, $x^{-1}(x\vee y)\in P_0$. We divide the proof into several cases.

\textbf{Case 1}: take $y=(1,1)$. It is clear that if we take $x=(1,1)$, then $x^{-1}(x\vee y)=(0,1)\in P^\ast$. Suppose we take $x=(0,p)$ for a prime $p$, then one can check that $x\vee y=(p,p)=x\cdot (1,1)$. Thus, $x^{-1}(x\vee y)=(1,1)\in P_0$. 

\textbf{Case 2}: Take $y=(i,p)$ for some prime $p$ and $0\leq i<p$. We divide the choices of $x$ into three cases:

When $x=(1,1)$, we have $x\vee y=(p,p)=(1,1)(p-1,p)$. Therefore, $x^{-1}(x\vee y)=(p-1,p)\in P_0$. 

When $x=(0,p)$, we have $x\vee y=\infty$ unless $y=(0,p)$, in which case $x^{-1}(x\vee y)=(0,1)\in P^\ast$. 

When $x=(0,q)$ for some prime $q\neq p$, we have $x\vee y=(\ell,pq)$, where $\ell=\min\{(i+p\mathbb{N})\cap q\mathbb{N}\}$. Notice that by the Chinese remainder theorem, there always exists a solution $\ell\in[0,pq-1)$, and thus $\ell=kq$ for some $0\leq k< p$. Hence, $x\vee y=(kq,pq)=(0,q)(k,p)$ and thus $x^{-1}(x\vee y)=(k,p)\in P_0$. This finishes the last case of the proof.  
\end{proof} 

Therefore, we obtain the following characterization:

\begin{theorem}\label{thm.main.NN} Let $T:\mathbb{N} \rtimes \mathbb{N}^\times\to\bh{H}$ be a contractive representation. Then, $T$ has a $\ast$-regular dilation if and only if for any $F\subseteq P_0=\{(1,1),(i,p): 0\leq i<p, p\mbox{ is a prime}\}$
$$\sum_{U\subseteq F} (-1)^{|U|} T(\vee U)T(\vee U)^* \geq 0.$$
\end{theorem}

\subsection{Baumslag-Solitar monoids} 

The Baumslag-Solitar monoid  $B_{n,m}$ (Example \ref{ex.Zappa} \pref{ex.Zappa.BS}) is the monoid generated by $a,b$ with the relation $ab^n=b^m a$. Each $B_{n,m}$ is a right LCM semigroup. 

\begin{lemma}\label{lm.BS} Every Baumslag-Solitar monoid $B_{n,m}$ has the descending chain property. The set of minimal elements is precisely $\{a,b\}$. 
\end{lemma}

\begin{proof} Every elements $p\in P$ can have many different expressions as product of $a,b$. We let $\ell(p)$ to be the maximum number of $a,b$ we can use to express $p$. $\ell(p)$ is always bounded \cite[Lemma 2.2]{Jackson2002}. It is clear that for any $p,q\in B_{n,m}$, $\ell(pq)\geq \ell(p)+\ell(q)$. Therefore, whenever $p,q\neq e$, we have $\ell(p),\ell(q)<\ell(pq)$. Since $\ell(p)\geq 1$ are integer-valued, $B_{n,m}$ has the descending chain property. It is clear that the set of minimal elements are $\{a,b\}$. 
\end{proof} 

We first find a minimal set for $B_{n,m}$. 

\begin{proposition}\label{pn.BS.minimal} $P_0=\{b^i a: 0\leq i\}\cup\{b^j: 1\leq j\}$ is a minimal set for $B_{n,m}$. 
\end{proposition} 

\begin{proof} We need to show for all $x\in P_{min}$ and $y\in P_0$, $x^{-1}(x\vee y)\in P_0$. We divide the proof into several cases.

\textbf{Case 1}. Suppose $y=b^i a$ for some $0\leq i$. If $x=a$, then either $i$ is a multiple of $m$ in which case $x^{-1}(x\vee y)=b^i\in P_0\cup P^\ast$, or $i\neq 0$ in which case $x\vee y=\emptyset$. If $x=b$, then either $i=0$ in which case $x^{-1}(x\vee y)=b^{m-1}a\in P_0$, or $i\neq 0$ in which case $x^{-1}(x\vee y)=b^{i-1} a\in P_0$. 

\textbf{Case 2}. Suppose $y=b^j$ for some $1\leq j$. If $x=b$, then $x^{-1}y=b^{j-1}\in P_0\cup P^\ast$. If $x=a$, then $x\vee y=b^\ell a$ where $\ell=min\{m\mathbb{N}\cap \mathbb{N}_{\geq j}\}$. Assume $\ell=km$, we have $x\vee y=a b^{kn}$. Hence, $x^{-1}(x\vee y)=b^{kn} \in P_0$. This finishes the proof. 
\end{proof} 

In fact, we can further reduce this set $P_0$ to a smaller set. Let $P_{00}=\{b, b^i a: 0\leq i\leq m-1\}$. 

\begin{proposition} The following are equivalent:
\begin{enumerate}
\item $Z(F)\geq 0$ for all finite $F\subset P_0$.
\item $Z(F)\geq 0$ for all finite $F\subset P_{00}$. 
\end{enumerate}
\end{proposition}

\begin{proof} It is clear that $P_{00}\subset P_0$ and thus one direction is trivial. Now suppose $Z(E)\geq 0$ for all finite $E\subset P_{00}$. Now take a finite $F\subset P_0$ and let $k(F)=\max\{i: b^i a\in F\}$ and $\ell(F)=\max\{j: b^j\in F\}$. We know $F\subset P_{00}$ when $k(F)<m$ and $\ell(F)\leq 1$. 

Suppose $k(F)=k\geq m$, then write $F=\{b^k a, p_2,\cdots, p_n\}$. Due to Lemma \ref{lm.base}, we may assume all other elements in $F$ with the form $b^i a$ has $i<k$. Denote 
\begin{align*}
F_1 &= \{b, p_2,\cdot, p_n\},\\
F_2 &= \{b^{k-1}a, b^{-1}(b\vee p_2), \cdots, b^{-1}(b\vee p_n)\}.
\end{align*}
It follows from the Lemma \ref{lm.reduction} that $Z(F)\geq 0$ if both $Z(F_1)\geq 0$ and $Z(F_2)\geq 0$.  Notice that we replaced $b^k a$ by $b$ in $F_1$, so that $k(F_1)<k(F)$ and $\ell(F_1)=\ell(F)$. For $F_2$, it follows from the calculation in the Proposition \ref{pn.BS.minimal} that if $p_i=b^i a$, then $$b^{-1}(b\vee p_i)=\begin{cases} b^{i-1} a, i\geq 1 \\ b^{m-1} a, i=0 \end{cases}$$
If $p_i=b^j$, then $b^{-1}(b\vee p_i)=b^{j-1}$. Therefore, $k(F_2)<\max\{k(F),m-1\}$ and $\ell(F_2)\leq\ell(F_1)$. 

Suppose $\ell(F)=\ell>1$, then write $F=\{b^\ell, p_2,\cdots, p_n\}$. Denote 
\begin{align*}
F_1 &= \{b, p_2,\cdot, p_n\},\\
F_2 &= \{b^{\ell-1}, b^{-1}(b\vee p_2), \cdots, b^{-1}(b\vee p_n)\}.
\end{align*}
It follows from the Lemma \ref{lm.reduction} that $Z(F)\geq 0$ if both $Z(F_1)\geq 0$ and $Z(F_2)\geq 0$. A similar computation shows that $k(F_1)=k(F)$, $k(F_2)\leq\max\{k(F),m-1\}$, and $\ell(F_1),\ell(F_2)<\ell(F)$. 

Combining these two cases, we able to repeatedly use Lemma \ref{lm.reduction} and induction on $(k(F),\ell(F))$ to show $Z(F)\geq 0$ assuming $Z(E)\geq 0$ for all finite $E\subset P_{00}$. 
\end{proof} 

\begin{theorem}\label{thm.main.BS} Let $B_{n,m}$ be a Baumslag-Solitar monoid for $n,m\geq 1$ and let $a,b$ be its generators. Let $P_{00}=\{b, b^i a: 0\leq i\leq m-1\}$. Then $T$ is $\ast$-regular if and only if $Z(F)\geq 0$ for all finite $F\subset P_{00}$. 
\end{theorem} 

%

\section{The Graph Product of Right LCM Semigroups}\label{sec.gp}

Let $\Gamma=(V,E)$ be a countable simple undirected graph (i.e. the vertex set $V$ is countable, and there is no 1-loop or multiple edges in the graph). Suppose $P=(P_v)_{v\in V}$ is a countable collection of right LCM semigroups. The graph product $\Gamma_{v\in V} P_v$ is the semigroup defined by taking the free product $\ast_{v\in V} P_v$ modulo the relation $p\in P_v$ commutes with $q\in P_u$ whenever $(u,v)$ is an edge in the graph $\Gamma$. For simplicity, we shall denote $P_\Gamma=\Gamma_{v\in V} P_v$. 

The graph product of groups was first studied in Green's thesis \cite{Green1990}. Subsequently, it is used to construct new quasi-lattice ordered groups \cite{CrispLaca2002}. It is shown that a graph product of quasi-lattice ordered groups is also quasi-lattice ordered \cite[Theorem 10]{CrispLaca2002}. This is generalized to graph products of right LCM semigroups where it is shown that a graph product of right LCM semigroups is still right LCM \cite[Theorem 2.6]{FK2009} (though the original statement concerns left LCM semigroups, this can be easily translated into right LCM semigroups). 

Regular dilation for representations of graph product of $\mathbb{N}$ is recently studied in \cite{Li2017}, where it unifies the Brehmer's dilation on $\mathbb{N}^k$ and the Frazho-Bunce-Popescu's dilation on $\mathbb{F}_k^+$. We would now like to extend this result further to graph product of right LCM semigroups.

Given $x\in P_\Gamma$, if we can write $x=x_1x_2\cdots x_n$ where each $e\neq x_j\in P_{v_j}$, this is called an expression of $x$. Each $x_j$ is called a syllable in the expression. For $e\neq p\in\bigcup_{v\in V} P_v$, let $I(p)=v$ if $p\in P_v$. 

Let $x=x_1x_2\cdots x_n$ be an expression of $x$. Suppose $I(x_j)$ is adjacent to $I(x_{j+1})$, then $x_j x_{j+1}=x_{j+1} x_j$ and thus we can write $$x=x_1\cdots x_{j-1} x_{j+1} x_j x_{j+2} \cdots x_n.$$
This is called a shuffle of $x$. Two expressions of $x$ are called shuffle equivalent if one expression can be obtained from the other via finitely many shuffles. 

In the case when $I(x_j)=I(x_{j+1})$, we can let $x_j'=x_j x_{j+1}$ and write $$x=x_1\cdots x_{j-1} x_j' x_{j+2} \cdots x_n.$$
This is called an amalgamation. 

An expression $x=x_1\cdots x_n$ is called a reduced expression for $x$ if it is not shuffle equivalent to an expression that admits an amalgamation. Equivalently, this implies whenever $I(p_i)=I(p_j)$ for some $i<j$, there exists $i<k<j$ so that $I(p_k)$ is not adjacent to $I(p_i)$. A result of Green \cite{Green1990} states that every element $x$ has a reduced expression, and any two reduced expression of $x$ are shuffle equivalent. Therefore, one can define the $\ell(x)$ to be the number of syllables in a reduced expression of $x$. $\ell(x)$ is the least number of syllables in an expression of $x$. By convention, if $x=e$, $\ell(x)=0$. 

Given a reduced expression $x=x_1x_2\cdots x_n$, a syllable $x_i$ is called an initial syllable if we can shuffle this reduced expression as $x=x_i x_2'\cdots x_n'$. Notice that we can shuffle $x_i$ to the front if and only if $x_i$ commutes with all the syllables $x_1,\cdots, x_{i-1}$. Therefore, if $x_i, x_j$ are two distinct initial syllables of $x$, they have to commute. We call a vertex $v$ an initial vertex of $x$ if there exists an initial syllable $x_i$ of $x$ with $I(x_i)=v$.

It is clear that when $x=x_1\cdots x_n$, $x_1$ is always an initial syllable of $x$ and $v=I(x_1)$ is an initial vertex. Moreover, even if the expression $x=x_1\cdots x_n$ is not a reduced expression, $I(x_1)$ is still an initial vertex of $x$ (as long as $x_1\neq e$). This follows from the fact that $y=x_2\cdots x_n$ admits a reduced expression $y_1\cdots y_k$, and $x=x_1\cdot y_1\cdots y_k$. Either $v$ is not an initial vertex of $y$ and $x_1$ is an initial syllable, or $v$ is an initial vertex of $y$ and $x_1$ amalgamate with this initial vertex and form an initial syllable from $P_v$. 

The graph product of right LCM semigroups has some nice properties. Let us fix a simple graph $\Gamma=(V,E)$ and a collection of right LCM semigroups $(P_v)_{v\in V}$. Let their graph product be $P_\Gamma$. The next two lemmas are directly taken from \cite{FK2009}. 

\begin{lemma}[{\cite[Lemma 2.5]{FK2009}}]\label{lm.gp1} Let $e\neq p\in P_u$ and $e\neq q\in P_v$ where $(u,v)\in E$. Then $$pP_\Gamma \cap qP_\Gamma = pq P_\Gamma.$$
\end{lemma} 

\begin{lemma}[{\cite[Lemma 2.7]{FK2009}}]\label{lm.gp2} Let $x,y\in P_v$ for some $v\in V$. Then
\begin{enumerate}
\item $xP_v\cap yP_v =\emptyset$ if and only if $xP_\Gamma \cap y P_\Gamma = \emptyset$.
\item If $xP_v\cap yP_v = zP_v$ (i.e. $z\in x\vee y$), then $xP_\Gamma\cap yP_\Gamma = z P_\Gamma.$
\end{enumerate}
\end{lemma}

Lemma \ref{lm.gp1} implies that for $e\neq p\in P_u$ and $e\neq q\in P_v$ where $(u,v)\in E$, $pq\in p\vee q$. Following the proof of \cite[Lemma 2.5]{FK2009}, one can deduce that this is true for more than $2$ vertices. Recall a finite subset $W\subset V$ is called a clique if every two vertices in $W$ are adjacent in $\Gamma$. 

\begin{lemma}\label{lm.gp.clique} If $W\subseteq V$ is a clique in $\Gamma$, and $e\neq p_v\in P_v$ for all $v\in W$. Then $\prod_{v\in W} p_v\in\vee\{p_v: v\in W\}$. In other words, $$\big(\prod_{v\in W} p_v\big) P_\Gamma =\bigcap_{v\in W} p_v P_\Gamma.$$
\end{lemma} 

We now prove a key technical lemma in our analysis of the $\ast$-regular condition on graph product of right LCM semigroup. 

\begin{lemma}\label{lm.gp.technical} Let $p\in P_v$ and $x\in P_\Gamma$ so that $x\vee p\neq\emptyset$. Then there exists $s\in x\vee p$ with $\ell(p^{-1} s)\leq \ell(s)$. 
\end{lemma}

\begin{proof} The statement is trivially true if $p=e$ since we can simply pick $s=x\in x\vee e$. Suppose otherwise, let $x=x_1 x_2\cdots x_n$ be a reduced expression of $x$ and let $x_1\in P_u$. Here, $\ell(x)=n$. Let $y=x_2\cdots x_n$ and $x=x_1 y$. First of all, since $x\vee p\neq\emptyset$, there exists $q=p\cdot p'=x\cdot x'=x_1 y\cdot x'$ for some $x',p'\in P_\Gamma$. Since $p$ and $x_1$ are in the front of this expression, $u,v$ are both initial vertices of $q$ and thus either $u=v$ or $(u,v)$ is an edge of $\Gamma$. 

Let us do an induction on $\ell(x)$. In the base case when $\ell(x)=1$, $x=x_1$ has only one syllable in its reduced expression. There are two cases: 

\underline{Case 1}: if $u=v$, then $x_1,p\in P_v$. By Lemma \ref{lm.gp2}, $xP_\Gamma\cap p P_\Gamma\neq\emptyset$ implies $xP_v\cap p P_v\neq\emptyset$. Therefore, we can pick $s\in x\vee p\subset P_v$ and $p^{-1} s\in P_v\subset P_\Gamma$. $p^{-1} s$ has only one syllable, and its length is either $0$ (when $p^{-1}s=e$) or $1$. Hence, $\ell(p^{-1}s)\leq 1= \ell(x)$. 

\underline{Case 2}: if $u\neq v$, then $(u,v)$ must be an edge of $\Gamma$. By Lemma \ref{lm.gp1}, we can pick $s=px_1\in p\vee x$ and thus $\ell(p^{-1} s)=\ell(x)$. 

Suppose now the statement holds true for all $x$ with $\ell(x)<n$. Now consider the case when $\ell(x)=n$ and $x=x_1\cdots x_n$ is an reduced expression of $x$. Let $y=x_2\cdots x_n$. There are again two cases.
 
\underline{Case 1}: if $u=v$, then $xP_\Gamma\cap p P_\Gamma\neq\emptyset$ implies $x_1 P_v\cap p P_v\neq\emptyset$. Pick $t\in x_1\vee p\in P_v$ and let $q=x_1^{-1} t\in P_v$. We first prove that $$tP_\Gamma\cap xP_\Gamma = pP_\Gamma\cap xP_\Gamma.$$

First, by Lemma \ref{lm.gp2}, $t\in x_1\vee p$ implies $tP_\Gamma=x_1 P_\Gamma\cap p P_\Gamma$. Hence $$tP_\Gamma\cap xP_\Gamma \subseteq pP_\Gamma\cap xP_\Gamma.$$
Conversely, by Lemma \ref{lm.gp2}, $$pP_\Gamma\cap xP_\Gamma\subset pP_\Gamma\cap x_1 P_\Gamma=t P_\Gamma.$$
This proves the other inclusion.

Now $t\vee x=p\vee x$. But $t=x_1 q$ and $x=x_1 y$, by Lemma \ref{lm.left.inv}, $x_1\cdot(q\vee y)=p\vee x\neq\emptyset$. In particular, $q\vee y\neq\emptyset$. Notice that $y$ is obtained by removing the initial syllable $x_1$ from a reduced expression $x=x_1y$. Hence, $\ell(y)=\ell(x)-1<n$. By the induction hypothesis, there exists $s\in q\vee y$ and $\ell(q^{-1} s)\leq\ell(y)$. 

Let $w=q^{-1} s$ and $s=qw\in q\vee y$. Let $s'=x_1 s\in x_1(q\vee y)=p\vee x$. $s'=x_1 q w=tw$ and $p^{-1}s'=(p^{-1}t)w$. The induction hypothesis gives $\ell(w)\leq \ell(y)$. Now $p^{-1}t\in P_v$ and thus $\ell(p^{-1} s')=\ell(p^{-1} t w)\leq \ell(w)+1$. Hence $$\ell(p^{-1} s')\leq\ell(w)+1\leq\ell(y)+1=\ell(x),$$
where $s'\in p\vee x$. This finishes the induction step for this case.

\underline{Case 2}: if $u\neq v$, then $(u,v)$ must be an edge of $\Gamma$ and $x_1,p$ commute. We first prove that $$x_1p P_\Gamma\cap xP_\Gamma=pP_\Gamma\cap xP_\Gamma.$$

The $\subseteq$ direction is trivial as $x_1pP_\gamma=px_1P_\Gamma\subset pP_\Gamma$. Conversely, by Lemma \ref{lm.gp1}, $$pP_\Gamma\cap xP_\Gamma\subset pP_\Gamma\cap x_1 P_\Gamma=x_1 p P_\Gamma.$$
This proves the other inclusion. 

By Lemma \ref{lm.left.inv}, $p\vee x=x_1p\vee x_1y=x_1(p\vee y)\neq\emptyset$. Hence $p\vee y\neq\emptyset$. Moreover, $\ell(y)=\ell(x)-1<n$. By the induction hypothesis, there exists $s\in p\vee y$ so that $\ell(p^{-1} s)\leq \ell(y)$. Let $w=p^{-1} s$ and $s'=x_1 s=x_1 p w\in x_1(p\vee y)=p\vee x$. Now $$\ell(p^{-1} s')=\ell(p^{-1} x_1 p w)=\ell(x_1 w)\leq \ell(w)+1\leq \ell(y)+1=\ell(x),$$
where $s'\in p\vee x$. This finishes the induction step for this case and thus the entire proof. \end{proof}  

Now consider a collection of representations $T_v:P_v\to\bh{H}$. Suppose for any edge $(u,v)$ of $\Gamma$, $T_u(p)$ commutes with $T_v(q)$ for all $p\in P_u$ and $q\in P_v$. Then we can build a representation $T:P_\Gamma\to\bh{H}$ where for any $x=x_1\cdots x_n$, $x_i\in P_{v_i}$, $$T(x)=T_{v_1}(x_1)T_{v_2}(x_2)\cdots T_{v_n}(x_n).$$

Since the commutation relations of $T_v$ coincide with the commutation relations in $P_\Gamma$, this defines a representation $T$ on the graph product $P_\Gamma$. In fact, every representation $T$ of $P_\Gamma$ arises in this way since we can simply let $T_v$ be the restriction of $T$ on $P_v$. We are interested in when the representation $T$ has $\ast$-regular dilation. 

\begin{example} Take $P_v=\mathbb{N}$ for all $v\in V$. This semigroup $P_\Gamma$ is the graph product of $\mathbb{N}$, also known as a right-angled Artin monoid as discussed previously (Example \ref{ex.quasi}\pref{ex.quasi.2}). Each representation $T_v$ of $P_v=\mathbb{N}$ is uniquely determined by the value $T_v=T_v(1_v)$. The commutation relations require that $T_u,T_v$ commute whenever $(u,v)$ is an edge of $\Gamma$. 

The $\ast$-regular dilation for such representation $T$ of graph product of $\mathbb{N}$ was the focus of \cite{Li2017}. A Brehmer-type condition is established in \cite[Theorem 2.4]{Li2017}. It is shown that the following are equivalent:

\begin{enumerate}
\item $T$ has a $\ast$-regular dilation;
\item $T$ has a minimal isometric Nica-covariant dilation;
\item For every finite $W\subset V$, $$\sum_{\substack{U\subseteq W \\ U\mbox{ is a clique}}} (-1)^{|U|} T_U T_U^*\geq 0.$$
Here, $T_U=\prod_{v\in U} T_v$.
\end{enumerate} 
\end{example} 

We would like to extend our result of $\ast$-regular dilation on graph product of $\mathbb{N}$ to graph product of any right LCM semigroup. We have derived in the Theorem \ref{thm.main} that $T$ has a $\ast$-regular dilation if and only if for every finite set $F\subset P_\Gamma$, $$Z(F)=\sum_{U\subseteq F} (-1)^{|U|} TT^*(\vee U)\geq 0.$$

The goal is to reduce $F$ further to a much smaller collection of subsets. 

\begin{proposition}\label{pn.gpLCM.main} Let $P_\Gamma$ be a graph product of a collection of right LCM smeigroups $(P_v)_{v\in V}$, and $T:P_\Gamma\to\bh{H}$ be a contractive representation. Then the following are equivalent:
\begin{enumerate}
\item\label{pn.gpLCM.main.1} For every finite set $F\subset P_\Gamma$, $Z(F)\geq 0$.
\item\label{pn.gpLCM.main.2} For every finite set $e\notin F\subset \bigcup_{v\in V} P_v$, $Z(F)\geq 0$. 
\end{enumerate}
\end{proposition}

\begin{proof} The direction \pref{pn.gpLCM.main.1} $\Rightarrow$ \pref{pn.gpLCM.main.2} is obvious. To show the converse, notice that a finite set $e\notin F\subset \bigcup_{v\in V} P_v$ if and only if every element $x\in F$ is inside some $P_v$ and thus $\ell(x)=1$ for all $x\in F$. Denote $c(F)=\sum_{x\in F} (\ell(x)-1)$. Then for a finite subset $e\notin F\subset P_\Gamma$, $c(F)\geq 0$ and $F\subset \bigcup_{v\in V} P_v$ if and only if $c(F)=0$. 

If $e\notin F$ has $c(F)>0$, then there exists $x\in F$ with $\ell(x)\geq 2$. Write $x=p_1 q$ for some $p_1\in P_v$ and $\ell(q)=\ell(x)-1$. Let $F=\{p_1q,p_2,\cdots,p_n\}$. Let 
\begin{align*}
F_1 &= \{p_1, p_2, \cdots, p_n\} \\
F_2 &= \{q, p_1^{-1} s_2, \cdots, p_1^{-1} s_n\}
\end{align*}
where $s_i\in p_1 \vee p_i$ for all $2\leq i\leq n$. By Lemma \ref{lm.reduction}, $Z(F)\geq 0$ if $Z(F_1)\geq 0$ and $Z(F_2)\geq 0$. 

Since $\ell(p_1)<\ell(p_1 q)$, we have $c(F_1)<c(F)$. By Lemma \ref{lm.gp.technical}, for each $2\leq i\leq n$, either $p_1\vee p_i=\emptyset$ or there exists $s_i\in p_1\vee p_i$ with $\ell(p_1^{-1} s_i)\leq \ell(p_i)$. Therefore, compare elements in $F$ with $F_2$: either an element $p_i$ is removed when $p_1\vee p_i=\emptyset$, or $p_i$ is replaced by $p_1^{-1} s_i$ with $\ell(p_1^{-1} s_i)\leq \ell(s_i)$. Moreover, the element $p_1 q$ in $F$ is replaced by $q$ where $\ell(q)=\ell(p_1 q)-1$. Hence, $c(F_2)<c(F)$. 

Now $c(F_1),c(F_2)<c(F)$. We can only repeat this process finitely many times. The positivity of any finite $e\notin F\subset P_\Gamma$ is therefore reduced to the positivity of sets of the form $F\subset P_\Gamma$ where $\ell(x)\leq 1$ for all $x\in F$. Notice that $Z(F)\geq 0$ whenever $e\in F$ (Lemma \ref{lm.base}\pref{lm.base.2}). Hence, the condition \pref{pn.gpLCM.main.2} is sufficient. \end{proof}

For a finite set $e\notin U\subset \bigcup_{v\in V} P_v$, we denote $I(U)=\{I(x): x\in U\}$. Suppose $(u,v)$ is not an edge of $\Gamma$, and $e\neq p\in P_u$, $e\neq q\in P_v$, then $pP_\Gamma\cap qP_\Gamma$ must be $\emptyset$ since $u,v$ are both initial vertices of any element $r\in p\vee q$. Therefore, $\vee U=\emptyset$ unless any two vertices in $U$ are adjacent to one another. In other words, $\vee U=\emptyset$ unless $I(U)$ is a clique in $\Gamma$. Hence, we can simplify $Z(F)$ as:
\begin{align*}
Z(F) &=\sum_{U\subseteq F} (-1)^{|U|} TT^*(\vee U) \\
&= \sum_{\substack{U\subseteq F \\ I(U)\mbox{ is a clique}}} (-1)^{|U|} TT^*(\vee U).
\end{align*}

Take a finite set finite $e\notin F\subset \bigcup_{v\in V} P_v$. Each $x\in F$ belongs to a certain copy of $P_v$. If $x=p_1 q$ with $p_1, q\in P_v$ and $F=\{p_1 q, p_2,\cdots, p_n\}$. Let $F_1,F_2$ be the subsets defined in Lemma \ref{lm.reduction}: 
\begin{align*}
F_1 &= \{p_1, p_2, \cdots, p_n\} \\
F_2 &= \{q, p_1^{-1} s_2, \cdots, p_1^{-1} s_n\}
\end{align*}
where $s_i\in p_1 \vee p_i$ for all $2\leq i\leq n$. For each $2\leq i\leq n$, apply the Lemma \ref{lm.gp.technical}, either $p_1\vee p_i=\emptyset$ or we can pick $s_i\in p_1\vee p_i$ with $\ell(p_1^{-1} s_i)\leq \ell(p_i)$. Hence, $F_1,F_2\subset\bigcup_{v\in V} P_v$. In the case when a semigroup $P_u$ satisfies the descending chain condition, the procedure described in the Proposition \ref{pn.desc.reduction} still applies. This can further reduce $F$ to a subset $F\subset\bigcup_{v\in V} P_v$ where every element in $F\cap P_u$ is a minimal element (i.e. $F\cap P_u=(P_u)_0$). 

Therefore, we obtain the following characterization of $\ast$-regular representations of a graph product of right LCM semigroups.

\begin{theorem}\label{thm.gp.main} Let $P_\Gamma$ be a graph product of right LCM semigroups and $T:P_\Gamma\to\bh{H}$ be a contractive representation. Then the following are equivalent:

\begin{enumerate}
\item $T$ has a $\ast$-regular dilation;
\item $T$ has a minimal isometric Nica-covariant dilation;
\item For every finite $F\subset P_\Gamma$, $$Z(F)=\sum_{U\subseteq F} (-1)^{|U|} TT^*(\vee U)\geq 0;$$
\item For every finite $e\notin F\subset \bigcup_{v\in V} P_v$, $$Z(F)=\sum_{\substack{U\subseteq F \\ I(U)\mbox{ is a clique}}} (-1)^{|U|} TT^*(\vee U)\geq 0.$$
\end{enumerate}

In particular, in the case when a $P_u$ satisfies the descending chain condition, we may assume $F\cap P_u\subset(P_u)_0$, where $(P_u)_0$ is the set of minimal elements of $P_u$. 
\end{theorem} 

\begin{example} In the case when $P_v=\mathbb{N}$ for all $v\in V$. $P_\Gamma$ is a graph product of $\mathbb{N}$, which is generated by the generator $\{e_v: v\in V\}$ with the relation $e_ue_v=e_ve_u$ whenever $(u,v)$ is an edge of $\Gamma$. Let $T:P_\Gamma\to\bh{H}$ be a contractive representation which is uniquely determined by $T_v=T(e_v)$. For each finite $W\subset V$, if $W$ is a clique, let $T_W=\prod_{v\in W} T_W$. Each $P_v=\mathbb{N}$ satisfies the descending chain condition, and the set of minimal elements $(P_v)_0=\{e_v\}$. 

By Theorem \ref{thm.gp.main}, $T$ has a $\ast$-regular dilation if and only if for every finite $F\subseteq\{e_v: v\in V\}$, $$Z(F)=\sum_{\substack{E\subseteq F \\ I(E)\mbox{ is a clique}}} (-1)^{|E|} TT^*(\vee E)\geq 0.$$

Notice that each finite $E\subseteq\{e_v: v\in V\}$ corresponds to a finite set $U=\{v: e_v\in E\}\subset V$. It is easy to see that $I(E)=U$ and $\vee E=\prod_{v\in U} e_v$. Therefore, $T$ has a $\ast$-regular dilation if and only if for every finite $W\subset V$, $$\sum_{\substack{U\subseteq W \\ U\mbox{ is a clique}}} (-1)^{|U|} T_U T_U^*\geq 0.$$
Here, $T_U=\prod_{u\in U} T(e_u)$. This gives an another proof of \cite[Theorem 2.4]{Li2017}.
\end{example} 

We would like to consider an application of Theorem \ref{thm.gp.main}. Let $\Gamma$ be a complete graph. In other words, $(u,v)\in E$ for all $u\neq v$ in $V$. The graph product $P_\Gamma$ is simply the direct sum $\oplus_{v\in V} P_v$. 

\begin{definition} A family of contractive representation $T_v:P_v\to\bh{H}$ are called doubly commuting if for any $u\neq v$ and $p\in P_u$, $q\in P_v$, $T_u(p)$ commutes with both $T_v(q)$ and $T_v(q)^*$. 

Suppose that $\Gamma$ is a complete graph. A representation $T:P_\Gamma\to\bh{H}$ is called doubly commuting if $T_v:P_v\to\bh{H}$ given by restricting $T$ on $P_v$ form a doubly commuting family of contractive representations. 
\end{definition} 

Doubly commuting representations on products of special semigroups have been previously studied. A doubly commuting representation of $\mathbb{N}^k$ is always regular (\cite{Brehmer1961}, see also \cite{VernBook} for an alternative proof using $C^*$-algebra and completely positive maps). Fuller \cite[Theorem 2.4]{Fuller2013} proved that a doubly commuting representation of $\oplus S_i$ is always regular, where $S_i$ is a countable additive subgroup of $\mathbb{R}^+$. We are now going to extend all these results to direct sums of right LCM semigroups. 

\begin{lemma}\label{lm.gp.doublyComm} Fix a finite subset $W\subset V$ and for each $w\in W$, $F_w\subset P_w$ is a finite subset. Let $F=\bigcup_{w\in W} F_w$. Let $T:P_\Gamma\to\bh{H}$ be a doubly commuting contractive representation and $T_v:P_v\to\bh{H}$ be the representation by restricting $T$ on $P_v$. Then $$TT^*(\vee F)=\prod_{w\in W} T_w T_w^*(\vee F_w)$$.

Moreover, let
\begin{align*}
Z(F) &= \sum_{U\subseteq F} (-1)^{|U|} TT^*(\vee U),\\
Z_w(F_w) &= \sum_{U_w \subseteq F_w} (-1)^{|U_w|} T_wT_w^*(\vee U_w).
\end{align*}
Then $\{Z_w(F_w)\}_{w\in W}$ is a collection of commuting operators, and $$Z(F)=\prod_{w\in W} Z_w(F_w).$$
\end{lemma} 

\begin{proof} For each $w\in W$, pick $p_w\in \vee F_w$. By Lemma \ref{lm.union}, $\vee F=\vee\{p_w\}_{w\in W}$. By Lemma \ref{lm.gp.clique}, $$\prod_{w\in W} p_w\in\vee\{p_w\}_{w\in W}=\vee F.$$
Hence,
\begin{align*}
TT^*(\vee F) &= TT^*(\prod_{w\in W} p_w) \\
&= \big(\prod_{w\in W} T_w(p_w)\big)\big(\prod_{w\in W} T_w(p_w)^*\big) \\
&= \prod_{w\in W} T_w(p_w) T_w(p_w)^*.
\end{align*}

Now for each $U\subseteq F$, let $U_w=U\cap F_w$ be disjoint subsets. Then, $$(-1)^{|U|} TT^*(\vee U)=\prod_{w\in W} (-1)^{|U_w|} T_wT_w^*(\vee U_w).$$

It is now easy to check $Z(F)=\prod_{w\in W} Z_w(F_w)$. 
\end{proof} 

As a direct consequence of Lemma \ref{lm.gp.doublyComm} and Theorem \ref{thm.gp.main}:

\begin{theorem} Let $T:P_\Gamma\to\bh{H}$ be a doubly commuting contractive representation of $P_\Gamma$ and $T_v:P_v\to\bh{H}$ be the representation by restricting $T$ on $P_v$. Then $T$ has a $\ast$-regular dilation if and only if each $T_v$ has a $\ast$-regular dilation as a representation of $P_v$.
\end{theorem}

\bibliographystyle{abbrv}
\bibliography{DilnQuasiLatNew1}

\end{document}